\documentclass[oneside, a4paper,11pt,reqno]{amsart}
\usepackage{amssymb,amsmath,am  sthm,dsfont}
\usepackage{graphics,graphicx}
\usepackage[T1]{fontenc}
\usepackage{color}

\newcommand{\Z}{{Z\!\!\!Z}}

\newtheorem{hypo}{Hypothesis}

\newtheorem{prop}[hypo]{Proposition}

\newtheorem{thm}[hypo]{Theorem}

\newtheorem{lem}[hypo]{Lemma}

\def\PP{\mathbb{P}}
\def\RR{\mathbb{R}}
\def\ZZ{\mathbb{Z}}

\def\EE{\mathbb{E}}
\def\NN{\mathbb{N}}
\def\e{\varepsilon}
\def\Bb{\overline{B}}
\def\Cb{\overline{C}}
\def\Lb{\overline{L}}
\def\Vb{\overline{V}}
\def\Sb{\overline{S}}
\def\Xb{\overline{X}}

\newcommand {\floor}[1] {\left\lfloor {#1} \right\rfloor}

\title[RWRE in a stratified oriented medium]
{Random walk in random environment in a two-dimensional stratified medium with
orientations}
\date{\today}

\author{Alexis Devulder}
\address{Universit\'e de Versailles Saint-Quentin-en-Yvelines, Laboratoire de Math\'ematiques de
Versailles, CNRS UMR 8100, Bât. Fermat,
45 avenue des Etats-Unis,
78035 Versailles Cedex, France.}
\email{devulder@math.uvsq.fr }

\author{Fran\c{c}oise P\`ene}
\address{Universit\'e Europ\'eenne de Bretagne, Universit\'e de Brest,
D\'epartement de Math\'ematiques, 29238 Brest cedex, France}
\email{francoise.pene@univ-brest.fr}

\subjclass[2010]{60F17; 60G52; 60K37}
\keywords{random walk on randomly oriented lattices, random walk in random environment,
random walk in random scenery, functional limit theorem, transience.\\
This research was supported by the french ANR project MEMEMO2 2010 BLAN 0125.}

\begin{document}

\begin{abstract}
We consider a model of random walk in ${\mathbb Z}^2$ with (fixed or random) orientation of
the horizontal lines (layers) and with non constant iid probability to stay on these lines.
We prove the transience of the walk for any fixed orientations under general
hypotheses. This contrasts with the model of Campanino and Petritis \cite{CP},
in which probabilities to stay on these lines are all equal.
We also establish a result of convergence in distribution for this walk
with suitable normalizations under more precise assumptions. In particular, our model proves to be,
in many cases, even more superdiffusive than the random walks introduced by Campanino
and Petritis.
\end{abstract}
\maketitle

\section{Introduction}

In this paper we consider a random walk $(M_n)_n$ starting from $0$ on an oriented version
of ${\mathbb Z}^2$.
Let $\e=(\e_k)_{k\in \mathbb Z}$ be a sequence of random variables with values in $\{-1,1\}$
and joint distribution $\mu$.
%We make no assumption on this sequence.
%We suppose that all the even horizontal lines are oriented
%to the right, all the uneven are oriented to the left.
We assume that the $k^{th}$ horizontal line is entirely oriented to the right if $\e_k=1$,
and to the left if $\e_k=-1$.
We suppose
that the probabilities $p_k$ to stay on the $k^{th}$ horizontal line are given
by a sequence of independent identically distributed random variables
$\omega=(p_k)_{k\in\mathbb Z}$ (with values in $(0,1)$ and joint distribution $\eta$)
and that the probability to
go up or down are equal.

More precisely, given $\e$ and $\omega$, the process $(M_n=(M_n^{(1)},M_n^{(2)}))_n$
is a Markov chain satisfying $M_0=(0,0)$ with transition probabilities given by~:
$${\mathbb P}^{\e,\omega}( M_{n+1}-M_n=(\e_{M_n^{(2)}},0) \left\vert M_0,...,M_n\right. )=
         p_{M_n^{(2)}} $$
and
$$\forall y\in\{-1,1\},\ \ \ \
   {\mathbb P}^{\e,\omega}
  ( M_{n+1}-M_n=(0,y)\left\vert M_0,...,M_n\right. )=\frac{1-
    p_{M_n^{(2)}} }2.$$

\begin{center}
\includegraphics[scale=0.43,angle=90,trim=100mm 0mm 30mm 0mm,clip]{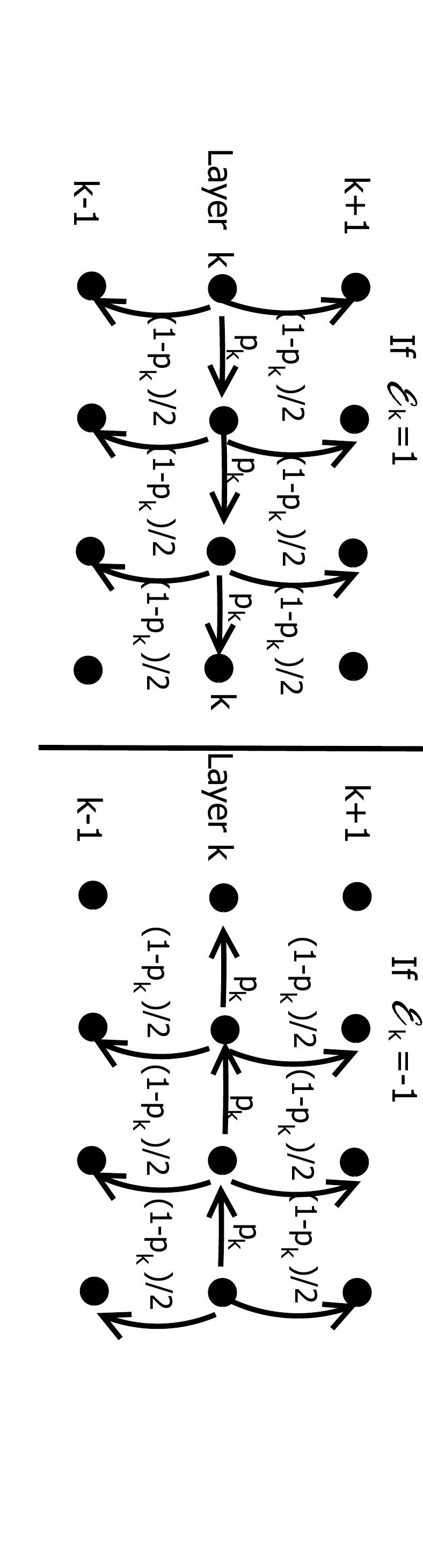}
\end{center}

We also define %${\mathbb P}^\omega$ and
the annealed probability $\mathbb P$ as follows:
$$
%{\mathbb P}^\omega(\cdot):=\int{\mathbb P}^{\e,\omega}(\cdot)\, d\mu(\e)\ \
%    \mbox{and}\ \
  \mathbb P(.):=\int\mathbb P^{\e,\omega}(.)d\eta(\omega)\, d\mu(\e).$$
We denote by $\mathbb E$ and ${\mathbb E}^{\e,\omega}$ the expectations with regard to
$\mathbb P$ and $\mathbb P^{\e,\omega}$ respectively.

Our model corresponds to a random walk in a two dimensional stratified medium
with oriented horizontal layers and
with random probability to stay on each layer.

The model with $p_k=1/2$ and with the $\varepsilon_k's$ iid and centered can be seen
as a discrete version of a model introduced by G. Matheron and G. de Marsily
in \cite{MdM} to modelize transport in a stratified porus medium.
This discrete model appears in \cite{BGKPR} to simulate the Matheron and de Marsily model.
It has also been introduced, separately, by mathematicians with motivations linked to quantum field theory
or propagation on large networks (see respectively \cite{CP} and \cite{CP2} and references therein).

In \cite{CP}, M. Campanino and D. Petritis proved that, when the $p_k$'s are all equal, the
behavior of the walk $(M_n)_n$ depends on the choice of the orientations $(\e_k)_k$.
First, they prove that the walk is recurrent
when $\e_k=(-1)^k$ (i.e. when the horizontal even lines are oriented
to the right and the uneven to the left). Second, they prove that the walk is
almost surely transient when the $\e_k's$ are iid and centered.
These results have been recently improved in \cite{CP2}.
Let us mention that extensions of this second model can be found in
\cite{GLN1,F1}, and that its Martin boundary is computed in \cite{deLoynes}.

In order to take into account the different nature of the successive layers of a stratified porus medium, it is natural to study the case where the $p_k$'s are random instead of being all equal. In this paper, we prove that taking the $p_k$'s random and i.i.d. can induce very different behaviors for the random walk.

First, we prove that under general hypotheses, the random walk is transient for every deterministic or random orientations, contrarily to the results obtained by Campanino and Petritis in \cite{CP} and \cite{CP2} for their model. Hence, even very small random perturbations of their (constant) $p_k$'s transform their recurrent walks into transient ones.

Second, it was proved in \cite{GLN2} that when the $p_k$'s are all equal, the random walk is superdiffusive, and that the horizontal position at time $n$ is, asymptotically,  of order $n^{3/4}$. This was conjectured in \cite{MdM} and was one main motivation for the introduction of this model. We prove that, depending on the law of $p_0$, our model can be even more superdiffusive, with horizontal position at time $n$ of order $n^\delta$, where $\delta$ can take all the values in $[3/4,1)$.

More precisely, our results are the following. We start by stating our theorem about transience.
\begin{thm}\label{sec:thm1}
Let $(p_k)_k$ be a sequence of independent identically distributed random variables.
Suppose here that $p_0$ is non-constant and that
${\mathbb E}[(1-p_0)^{-\alpha}]<\infty$ (for some $\alpha>1$).
Then, for every deterministic or random sequence $(\e_k)_k$, the random walk $(M_n)_n$ is transient
for almost every $\omega$.
\end{thm}

%Notice, in particular, that this result is true for any deterministic sequence of orientations $(\e_k)_k$.

We now give a functional theorem under more precise hypotheses. In particular,
we will assume that $\frac{p_0}{1-p_0}$ is integrable and that
the distribution of $\frac{p_0}{1-p_0}-{\mathbb E}\left[\frac{p_0}{1-p_0}\right]$ belongs to the
normal domain of attraction of a strictly stable distribution $G_\beta$
of index $\beta\in(1,2]$,
which means that
\begin{equation}\label{HypotheseAttraction1}
\PP\left(n^{-1/\beta}\sum_{k=1}^n \left(\frac{p_k}{1-p_k}-\EE\left[\frac{p_0}{1-p_0}
     \right]\right)\leq x\right)\to_{n\to+\infty}G_{\beta}(x),\qquad x\in\RR,
\end{equation}
the characteristic function $\zeta_\beta$ of $G_\beta$ being of the form
\begin{equation}\label{HypotheseAttraction2}
\zeta_\beta(\theta):=\exp[-|\theta|^{\beta}(A_1+i A_2\textnormal{sgn}(\theta))],\qquad \theta\in\RR,
\end{equation}
with $A_1>0$ and $|A_1^{-1} A_2|\leq |\tan(\pi \beta/2)|$.
Notice that this is possible iff $A_2=A_1\tan(\pi \beta/2)$
(since $\frac{p_0}{1-p_0}\geq 0$ a.s., see e.g. \cite[thm 2.6.7]{IL}).

If $\beta\in(1,2)$, we consider two independent right continuous stable processes
$(Z_x,\ x\geq 0)$ and $(Z_{-x},\ x\geq 0)$,  with characteristic functions
$$\EE(e^{i \theta Z_t})  =\exp[-A_1|t||\theta|^{\beta}], \qquad t\in\RR, \theta\in\RR.
$$
If $\beta=2$, we denote by $Z$ a two-sided standard Brownian motion.
We also introduce a standard Brownian motion $(B_t,\ t\ge 0)$,
and denote by $(L_t(x),\ x\in\RR,\  t\geq 0)$ the jointly continuous version
of its local time. We assume that $Z$ and $B$ are defined in the same probability space and are independent processes. We now define, as in \cite{KestenSpitzer}, the continuous process
$$\Delta_t:=\int_{\mathbb R}L_t(x)\, dZ_x, \qquad t\geq 0.$$
We prove the following result.

\begin{thm}\label{sec:TL}
Let $(p_k)_k$ be a sequence of independent identically distributed random variables
with values in $(0,1)$.
Suppose here that
%$p_0$ is non-constant,
${\mathbb E}\left[\frac{p_0}{1-p_0}\right]<\infty$
and that the distribution of
$\frac{p_0}{1-p_0}-{\mathbb E}\left[\frac{p_0}{1-p_0}\right]$
belongs to the normal domain of attraction of a strictly
stable distribution of index $\beta\in(1,2]$
(i.e. that we have \eqref{HypotheseAttraction1} and \eqref{HypotheseAttraction2}).

We also assume that $(\e_k)_k$ satisfies one of the following hypotheses~:
\begin{itemize}
\item[(a)] for every $k$, $\e_k= (-1)^k$,
\item[(b)] $(\e_k)_k$ is a sequence of independent identically distributed centered random
variables
with values in $\{\pm 1\}$; $(\e_k)_k$ is independent of $(p_k)_k$.
\end{itemize}
Then, setting $\delta:=\frac{1}{2}+\frac{1}{2\beta}$, the sequence of processes
$$\left(\left(n^{-\delta} M_{\floor {nt}}^{(1)},n^{-1/2} M_{\floor{nt}}^{(2)}\right)_{t\ge 0}\right)_n$$
converges in distribution under the probability $\PP$ (in the space of Skorokhod ${\mathcal D}([0;+\infty),{\mathbb
  R}^2)$) to
$(\gamma^{-\delta}\sigma\Delta_t,\gamma^{-1/2}B_t)_{t\ge 0}$ with $\gamma:=1+{\mathbb E}
\left[\frac{p_0}{1-p_0}\right]$ and with~:
\begin{itemize}
\item[*] $\sigma=\left(Var
\left(\frac{p_0}{1-p_0}\right)\right)^{1/2}$ in case {\it (a)} with $\beta=2$,
\item[*]$\sigma=\left({\mathbb E}\left[\left(\frac{p_0}{1-p_0}\right)^2\right]\right)^{1/2}$
in case {\it (b)} with $\beta=2$,
\item[*] $\sigma=1$ in cases (a) or (b) with $\beta\in(1,2)$.
\end{itemize}
\end{thm}

We remind that $\frac{p_0}{1-p_0}$ has a finite variance if $\beta=2$ (see e.g. \cite[Thm 2.6.6.]{IL}), hence $\sigma$ is finite in all cases.

%We notice in particular that the random walk $(M_n)_n$ can be even more superdiffusive
%than the model with $(p_k)_k$ constant (see \cite{GLN2}),
%since $\delta$ can take all the values in $[3/4,1)$.

The proof of this second result is based on the proof
of the functional limit theorem established by N.
Guillotin and A. Le Ny \cite{GLN2} for the walk of
M. Campanino and D. Petritis (with $(p_k)_k$ constant
and $(\e_k)_k$ centered, independent and identically distributed).

It may be possible that the transience remains true for every non degenerate distribution of the $p_k$'s on $(0,1)$. Indeed, roughly speaking, taking the $p_k$'s closer to one should make the random walk even more transient; however this is just an intuition and not a mathematical evidence.
We prove our Theorem \ref{sec:thm1} under a very general moment condition, which covers all the cases of our Theorem \ref{sec:TL}. In particular, the most superdiffusive cases, with $\delta>3/4$, are obtained when the support of  $1/(1-p_0)$ is not compact.

The proof of our first result is built from the proof of \cite[Thm 1.8]{CP}
with many adaptations. The idea is to prove that, when $(\e_k)_{k\in\Z}$ is a fixed sequence of orientations, that is when $\mu$ is a Dirac measure,
\begin{equation}
\sum_{k\ge 1} {\mathbb P} (M_k=(0,0))<+\infty.
\end{equation}
In the model we consider here, contrarily to the models envisaged in \cite{CP},
the second coordinate of $(M_n)_n$ is not a random walk but it is a random walk
in a random environment, since the probability
to stay on a horizontal line depends on the line, which complicates the model.
Even if
a central limit theorem and functional limit theorem have been established
in \cite{HWW} and in \cite{H} for $M_n^{(2)}$, the local limit theorem for
$M_n^{(2)}$ has not already been proved, to the extent of our knowledge.
Moreover, in Theorem \ref{sec:thm1} we do not assume that the distribution of
$\frac{p_0}{1-p_0}$ belongs
to the domain of attraction of a stable distribution. For these reasons, it does not seam
simple to make a precise estimation
of ${\mathbb P}(M_n=(0,0))$ as it has been done in \cite{BFFN}.
We also mention that the random walk $(M_n)_n$ is not reversible.

It will be useful to observe that under $\PP^{\e,\omega}$ and $\PP$, $(M_{T_n}^{(2)})_n$
is a simple random walk $(S_n)_n$ on $\mathbb Z$,
where the $T_n$'s are the times
of vertical displacement~:
$$T_0:=0; \ \forall n\ge 1,\ \ T_{n}:=\inf\{k>T_{n-1}\ :\ M_k^{(2)}\ne M_{k-1}^{(2)}\}. $$
We will use several times the fact that there exists $M>0$ such that, for
every $n\ge 1$, we have ${\mathbb P}(S_n=0)\le M n^{-\frac 12}$.
Now, let us write $X_n$ the first coordinate of $M_{T_n}$.
We observe that
$$X_{n+1}-X_{n}= \e_{S_n}\xi_n,$$
where $\xi_n:=T_{n+1}-T_n-1$ corresponds to the duration of the stay
on the horizontal line $S_n$ after the $n$-th change
of line.
Moreover, given $\omega=(p_k)_{k\in\mathbb Z}$, $\e=(\e_k)_{k\in\Z}$ and $S=(S_k)_k$,
the $\xi_k$'s are independent and with distribution given by
${\mathbb P}^{\e,\omega}(\xi_k=m|S)=(1-p_{S_k})p_{S_k}^m$ for every
$k\ge 0$ and $m\ge 0$.
With these notations, we have
$$X_{n}=\sum_{k=0}^{n-1}\e_{S_k}\xi_{k}.$$
%$$X_{2n}=\sum_{k=0}^{n-1}\left(\xi_{S_{2k}}^{(2k)}- \xi_{S_{2k+1}}^{(2k+1)}\right).$$

This representation of $(M_{T_{n}})_n$ will be very useful in the proof of both  the results.

\section{Estimate of the variance}
To point out the difference between our model and the model with $(p_k)_k$ constant
considered by M. Campanino and D. Petritis in \cite{CP}, we start by estimating
the variance of $X_{2n}$ under the probability $\PP$ for these two models
in the particular case when $\e_k=(-1)^k$ for every $k\in\mathbb Z$
and when $(1-p_0)^{-1}$ is square integrable.
\begin{prop}\label{sec:variance}
Let $\e_k=(-1)^k$ for every $k\in\mathbb Z$.
\begin{enumerate}
\item If the $p_k$'s do not depend on $k$, then $Var(X_{2n})={\mathbb E}
\left[\frac {2p_0}{(1-p_0)^2}\right]n$.
\item If the $(1-p_k)^{-1}$'s are iid, square integrable with positive variance,
then there exists $C>0$ such that $Var (X_{2n})\sim_{n\rightarrow +\infty} Cn^{3/2}$.
\end{enumerate}
\end{prop}
\begin{proof}[Proof of Proposition \ref{sec:variance}]
We observe that
$${\mathbb E}^{\e,\omega}[\xi_{k}|S]=\frac{p_{S_k}}{1-p_{S_k}}
\ \mbox{and}\ Var^{\e,\omega}(\xi_k|S)=\frac{p_{S_k}}{(1-p_{S_k})^2}.$$
Moreover, $S$ is independent of $\omega$ under $\PP$, hence ${\mathbb E}(X_{2n})=0$.
We have

$\displaystyle    Var(X_{2n})=\sum_{k,\ell=0}^{n-1}
{\mathbb E}\left[ \left( \xi_{2k}- \xi_{2k+1}  \right)
   \left( \xi_{2\ell}- \xi_{2\ell+1}  \right)
    \right]$
\begin{eqnarray*}
&=& \sum_{k=0}^{n-1}
{\mathbb E}\left[ \left( \xi_{2k}- \xi_{2k+1} \right)^2 \right]
   + 2\sum_{0\le k< \ell\le n-1}
{\mathbb E}\left[ \left( \xi_{2k}- \xi_{2k+1}  \right)
   \left( \xi_{2\ell}- \xi_{2\ell+1}  \right)
    \right]\\
&=&  \sum_{k=0}^{n-1}
{\mathbb E}\left[ \frac{p_0}{(1-p_0)^2}  +   \frac{p_1}{(1-p_1)^2}
         +\left(\frac {p_0}{1-p_0}-\frac {p_1}{1-p_1}\right)^2 \right]
   + 2\sum_{k=1}^{n-1} (n-k)
{\mathbb E}  \left[ \left( \xi_{0}- \xi_{1}  \right)
   \left( \xi_{2k}- \xi_{2k+1}  \right)
    \right]\\
&=& Cn    + 2\sum_{k=1}^{n-1} (n-k)
{\mathbb E}  \left[ \left( \frac {p_{S_{0}}}{1-p_{S_0}} - \frac { p_{S_{1}}}{1-p_{S_1}}  \right)
          \left( \frac {p_{S_{2k}}} {1-p_{S_{2k}} } - \frac { p_{S_{2k+1}}}{1-p_{S_{2k+1}}}  \right)
    \right].
\end{eqnarray*}
This gives the result in case (1). Now, to prove the result in case (2), we notice that,
since $p_y$ and $p_{y'}$ are independent as soon as $y\ne y'$, we have

$\displaystyle{\mathbb E}  \left[ \left( \frac {p_0}{1-p_0} - \frac {p_{S_1}}{1-p_{S_1}}  \right)
          \left( \frac {p_{S_{2k}} }{1-p_{S_{2k}} } - \frac{ p_{S_{2k+1}} }{1-p_{S_{2k+1}}}  \right)\right]$
\begin{eqnarray*}
&=& {\mathbb E}  \left[ \frac {p_0}{1-p_0}\frac {p_{S_{2k}}}{1-p_{S_{2k}}}+
       \frac {p_{S_1}} {1-p_{S_1}}
         \frac {p_{S_{2k+1}}}{1-p_{S_{2k+1}}} \right] -
         2 {\mathbb E}\left[\frac {p_0}{1-p_0}\right]^2\\
&=& 2\left( {\mathbb E}  \left[  \frac {p_0}{1-p_0}\frac {p_{S_{2k}}}{1-p_{S_{2k}}}\right]
     -{\mathbb E}\left[\frac {p_0}{1-p_0}\right]^2   \right)\\
&=&  2\left( {\mathbb E}  \left[  \frac {p_0^2}{(1-p_0)^2} \right]
     -{\mathbb E}\left[\frac {p_0}{1-p_0}\right]^2   \right){\mathbb P}(S_{2k}=0)\\
\end{eqnarray*}
\begin{eqnarray*}
&=&  2 Var\left( \frac {p_0}{1-p_0} \right)  {\mathbb P}(S_{2k}=0).
\end{eqnarray*}
We conclude as H. Kesten and F. Spitzer did in
\cite[p. 6]{KestenSpitzer}, using the fact that
${\mathbb P}(S_{2k}=0)\sim c k^{-1/2}$ (as $k$ goes to infinity) for some $c>0$.
\end{proof}
\section{Proof of Theorem \ref{sec:thm1} (transience)}
We come back to the general case.
It is enough to prove the result for any fixed $(\e_k)_{k}$.
Let $(\e_k)_{k\in\mathbb{Z}}$ be some fixed sequence of orientations. Hence $\mu$ is a Dirac measure on $\{-1,1\}^\Z$.
%and ${\mathbb P}^{\omega}={\mathbb P}^{\e,\omega}$.
Without any loss of generality, we assume throughout the proof of Theorem
\ref{sec:thm1} that $\e_0=1$ and $\alpha\leq 2$.
We have
$$ \sum_{k\ge 1} {\mathbb P} (M_k=(0,0))
    = \sum_{n\ge 1} {\mathbb P} (S_{2n}=0\ \mbox{and}
      \  X_{2n} \le 0\le X_{2n+1}  ).$$
Hence, to prove the transience, it is enough to prove that
\begin{equation}\label{eqSommePourTheoremeTransience}
\sum_{n\ge 1} {\mathbb P} (S_{2n}=0\ \mbox{and}
      \  X_{2n} \le 0\le X_{2n+1}  )<+\infty.
\end{equation}
This sum is divided into 8 terms which are separately estimated in
Lemmas \ref{lemA_n}, \ref{lemE_0}, \ref{lem}, \ref{lemE_2}, \ref{lemE_3},
\ref{lemB_n}, \ref{lemE_4} and \ref{lemE_1} provided $\delta_0,\delta_1,\delta_2,\delta_3$
are well chosen. One way to choose these $\delta_i$ so that they satisfy simultaneously the hypotheses of all these lemmas is given at the end of this section.

For every $y\in\mathbb Z$ and $m\in\NN$, we define $N_{m}(y):=\#\{k=0,...,m-1\ :\ S_k=y\}$.
We will use the fact that $X_{2n}=\overline{S}_{2n}+D_{2n}$ with
$$D_{2n}:=\sum_{y\in\mathbb Z}\frac{\e_y p_y} {1-p_y} N_{2n}(y) \ \ \mbox{and}\ \
  \overline{S}_{2n}:=\sum_{k=0}^{2n-1} \e_{S_k}\left(\xi_k-\frac{p_{S_k}}{1-p_{S_k}}\right) .$$
Roughly speaking, the idea of the proof is that
$X_{2n} \le 0\le X_{2n+1}$ implies that $X_{2n}$ cannot be very far away from 0,
which means
that $D_{2n}$ and $\overline{S}_{2n}$ should be of the same order, but this is false
with a large probability.
More precisely, we will prove that, with a large probability,
we have $|D_{2n}|>n^{\frac 3 4-\delta_3}$
and $|\overline{S}_{2n}|<n^{\frac 1 4+\frac 1 {2\alpha} + \upsilon}$ for small
$\delta_3>0$ and $\upsilon>0$ (see the definition of $B_n$ and the end of
the proof of Lemma \ref{lem}).
Now let us carry out carefully this idea.

Let $n\ge 1$.
Following \cite{CP}, we consider $\delta_1>0$ and $\delta_2>0$
and we define~:
$$\boxed{A_n:=\left\{ \max_{0\le k\le 2n}\vert S_{k}\vert\le n^{\frac 12 +\delta_1}
\ \ \mbox{and}\ \ \max_{y\in\mathbb Z}N_{2n}(y)<n^{\frac 12+\delta_2}\right\}}. $$
Our first lemma is standard, we give a proof for the sake of completeness.
\begin{lem}\label{lemA_n}
\begin{equation}\label{sec:EQ1}
\sum_{n\ge 0} {\mathbb P}(A_n^c)<+\infty.
\end{equation}
\end{lem}
\begin{proof}
Let $p>1$.
Thanks to Doob's maximal inequality and since $\EE(|S_n|^p)=O(n^{p/2})$, we have ${\mathbb E}[\max_{0\le k\le 2n}|S_k|^p]
=O(n^{\frac p 2})$ and so, by the Chebychev inequality,
$${\mathbb P}\left(\max_{0\le k\le 2n}|S_k|>n^{\frac 12+\delta_1}\right)\le
    \frac{{\mathbb E}[\max_{0\le k\le 2n}|S_k|^p]}{n^{p(\frac 12 +\delta_1)}}
    =O(n^{-p\delta_1}). $$
According to \cite[Lem. 1]{KestenSpitzer}, we also have
$\max_y{\mathbb E}[N_{2n}(y)^p]=O(n^{\frac p 2})$ and hence
$${\mathbb P}\left(\max_y N_{2n}(y)>n^{\frac 12+\delta_2}\right)\le
   \sum_{y=-2n}^{2n} {\mathbb P}( N_{2n}(y)>n^{\frac 12+\delta_2})
    =O(n^{1-p\delta_2}).$$
The result follows by taking $p$ large enough.
\end{proof}
Let
$\delta_0>0$ and set
$$\boxed{E_0(n):=\{p_0\leq 1-1/n^{\frac{1}{2\alpha}+\delta_0}\}}.$$
We have
\begin{lem}\label{lemE_0}
\begin{equation}
\sum_{n\ge 0}\PP(S_{2n}=0, E_0(n)^c) <+\infty.
\end{equation}
\end{lem}
\begin{proof}
Indeed, since $S$ is independent of $(p_k)_{k\in\Z}$, we have
\begin{equation*}
\PP(S_{2n}=0, E_0(n)^c)\leq \frac{M}{\sqrt{n}}\PP\left(\frac{1}{1-p_0}
>n^{\frac{1}{2\alpha}+\delta_0}\right)
   \leq \frac{M}{n^{1+\delta_0\alpha}}\EE\left[\left(\frac{1}{1-p_0}\right)^{\alpha}\right]
\end{equation*}
whose sum is finite.
\end{proof}
We also consider the conditional expectation of $X_{2n}$ with respect
to $(\omega,(S_p)_p)$ which is equal to
$D_{2n}=\sum_{y\in\mathbb Z}\frac{\e_y p_y} {1-p_y} N_{2n}(y)$.
We introduce $\delta_3>0$ and
$$
\boxed{B_n:=\left\{\left\vert D_{2n}
\right\vert> n^{\frac 34-\delta_3} \right\}}.
$$
Let $c_n:=n^{\frac{1}{\alpha}(\frac{1}{2}+\delta_1)+\delta_0}$ and
$$\boxed{E_1(n):=\left\{\forall y\in\{-n^{1/2+\delta_1}, n^{1/2+\delta_1}\},\
\frac{1}{1-p_y}\leq c_n\right\}}.$$
Since $p_0\in(0,1)$ a.s.,
there exist $0<a<b<1$ such that $\PP(a<p_0<b)=:\gamma_0>0$. Let
\begin{equation*}
\Lambda_n:=
\{k\in\{0,\dots,2n-1\},\ a<p_{S_k}<b\},
\end{equation*}
$P:=\{y\in\Z, a<p_y<b\}$, and $\zeta_y:=\mathbf{1}_{\{a<p_y<b\}}$, $y\in\Z$. We have
$\#\Lambda_n=\sum_{y\in\Z}\zeta_yN_{2n}(y)=\sum_{y\in P} N_{2n}(y)$.
Let
$$\boxed{ E_2(n):=\{ \#\Lambda_n \geq \gamma_0 n \} }.$$
Define $\Vb_{2n}:=\left(\sum_{k=0}^{2n-1} \left(\xi_{k}-\frac{p_{S_k}}{1-p_{S_k}}\right)
     ^2\right)^{1/2}$ and
$$\boxed{E_3(n):=\left\{\Vb_{2n}^2\leq n^{d+\delta_0}\right\}},$$
with $d:=\frac 12+\frac 1\alpha+3\delta_0+\frac{2\delta_1}\alpha+\delta_2$ and
$$\boxed{E_4(n):=\left\{\sum_{k=0}^{2n-1}\frac{1}{(1-p_{S_k})^2}\leq n^{d}\right\}}.$$

\begin{lem}\label{lem}
If
%$\frac{\delta_1}\alpha+3\delta_0< \frac 34-\frac 1{2\alpha}$ and
$\delta_3+\frac{\delta_1}\alpha+\frac{\delta_2}2+3\delta_0<\frac 12-\frac 1{2\alpha}$,
then we have
\begin{equation} \label{sec:EQ2}
\sum_{n\in\mathbb{N}}{\mathbb P}\left(
S_{2n}=0; \
X_{2n}\le 0\le X_{2n+1},\ A_n,B_n, \cap_{i=0}^3 E_i(n) \right)<\infty.
\end{equation}
\end{lem}
\begin{proof}
Uniformly on $E_0(n)\cap E_1(n)$, we have

$\displaystyle {\mathbb P}^{\e,\omega} (S_{2n}=0\ \mbox{and}
      \  X_{2n} \le 0\le X_{2n+1},A_n,B_n, E_2(n), E_3(n))$
\begin{eqnarray}
&\le &  \sum_{k\ge 0}  {\mathbb P}^{\e,\omega}(S_{2n}=0\
    \mbox{and}\ X_{2n}= -k,A_n,B_n, E_2(n), E_3(n))  (1-n^{-1/(2\alpha)-\delta_0})^k\nonumber\\
  &\le&  \sum_{k= 0}^{ n^{1/(2\alpha)+2\delta_0} }  {\mathbb P}^{\e,\omega}(S_{2n}=0\
    \mbox{and}\ X_{2n}= -k,A_n,B_n, E_2(n), E_3(n) )  + O(n^{-2})\nonumber\\
  &\le& {\mathbb P}^{\e,\omega}(S_{2n}=0\
    \mbox{and}\ -n^{1/(2\alpha)+2\delta_0}\le X_{2n}\le 0,A_n,B_n, E_2(n), E_3(n) )
 + O(n^{-2}).\label{eqEncadrementX}
\end{eqnarray}
In order to apply an inequality due to S.V. Nagaev \cite{Nagaev}, we define
$\Xb_k:=\e_{S_k}\left(\xi_k-\frac{p_{S_k}}{1-p_{S_k}}\right)$, recall that
$\Sb_{2n}=\sum_{k=0}^{2n-1} \Xb_k$, and introduce $\Bb_{2n}:=\left(\EE^{\e,\omega}\left[\Sb_{2n}^2\vert
S\right]
\right)^{1/2}=
\left(\sum_{j=0}^{2n-1} \frac{p_{S_j}}{(1-p_{S_j})^2}\right)^{1/2}$.
We have
$$\Bb_{2n}^2\ge a\sum_{y\,:\, p_y\geq a} \frac{N_{2n}(y)}{(1-p_{y})^2}.$$
Let
$\Cb(2n):=\sum_{k=0}^{2n-1}\EE^{\e,\omega}\left[\left|\Xb_k\right|^3\vert S\right]$.
On  $A_n\cap E_1(n)\cap E_2(n)$, we have
\begin{equation}\label{eqLdebut}
\sum_{y\ :\ p_y<a} \frac{ N_{2n}(y)}{(1-p_y)^2}  \le
    \frac{2n-\# \Lambda_n} {(1-a)^2}
   \le   \frac{2-\gamma_0}{\gamma_0} \frac{\# \Lambda_n} {(1-a)^2}
   \le \frac{2-\gamma_0}{\gamma_0} \sum_{y\ :\ p_y\geq a} \frac{ N_{2n}(y)}{(1-p_y)^2}
\end{equation}
and so
$$
\Cb(2n)
\leq \sum_{y} \frac{16 N_{2n}(y)}{(1-p_{y})^3}
\leq \sum_{y} \frac{16 N_{2n}(y)}{(1-p_{y})^2} c_n
\leq\frac{32}{\gamma_0} \sum_{y\ :\ p_y\geq a} \frac{N_{2n}(y)}{(1-p_y)^2}c_n.
$$
Let $\Lb_{2n}:=\Cb(2n)/\Bb_{2n}^3$.
On $A_n\cap E_1(n)\cap E_2(n)$,
we have
\begin{eqnarray}
\Lb_{2n} &\leq& \frac{32}{\gamma_0 a^{3/2}}\left(
    \sum_{y\ :\ p_y\geq a} \frac{ N_{2n}(y)}{(1-p_y)^2}  \right)^{-1/2} c_n
  \leq  \frac{32}{\gamma_0 a^{3/2}} \frac{1-a}{\sqrt{\gamma_0 n}} c_n\nonumber\\
  &\le& \frac{32(1-a)}{(\gamma_0 a)^{3/2}} n^{-\frac 12+\frac 1{2\alpha}+\frac
     {\delta_1}\alpha+\delta_0}\le n^{-2\delta_0},\label{eqLfin}
\end{eqnarray}
if $n$ is large enough,
since $\frac{\delta_1}\alpha+3\delta_0< \frac 12-\frac 1{2\alpha}$.

Let us recall that $\Vb_{2n}=\left(\sum_{k=0}^{2n-1}\Xb_k^2\right)^{1/2}$.
We can now apply Nagaev (\cite{Nagaev}, Thm 1),
which gives uniformly on $A_n\cap E_1(n)\cap E_2(n)$,
\begin{eqnarray}
{\mathbb P}^{\e,\omega}(|\Sb_{2n}|\geq n^{\delta_0} \Vb_{2n}\vert S)
& \leq &  2\left(\frac{n^{2\delta_0}}{4\log 2}+1\right)\exp\left(-\frac{n^{2\delta_0}}{4}(1-c' \Lb_{2n} n^{\delta_0})\right)+2\exp\left(-\frac{c''}{\Lb_{2n}^2}\right)
\nonumber\\
&=&O(\exp(-n^{\delta_0}))\label{InegNagaev}
\end{eqnarray}
where $c'>0$ and $c''>0$ are universal constants.
We recall that
$X_{2n}=\sum_{k=0}^{2n-1} \e_{S_k} \xi_k=\Sb_{2n}+D_{2n}$.
We have, for large $n$, on $A_n \cap B_n\cap E_1(n)\cap E_2(n)$,
\begin{eqnarray*}
{\mathbb P}^{\e,\omega}( -n^{\frac{1}{2\alpha}+2\delta_0} \le X_{2n}\le 0, E_3(n)\vert S)
&\le& {\mathbb P}^{\e,\omega}\left(\vert X_{2n}-D_{2n}\vert
   \ge n^{\frac 34-\delta_3}-n^{\frac{1}{2\alpha}+2\delta_0},E_3(n) \vert S\right)\\
&\leq&{\mathbb P}^{\e,\omega}(|\Sb_{2n}|\geq n^{\delta_0} n^{\left(d+\delta_0\right)/2}, E_3(n)\vert S)\\
&\leq&{\mathbb P}^{\e,\omega}(|\Sb_{2n}|\geq n^{\delta_0} \Vb_{2n}\vert S),
\end{eqnarray*}
since $\frac 1{2\alpha}+2\delta_0<\frac 34-\delta_3$ and since
$\delta_3+\frac{\delta_1}\alpha+\frac{\delta_2}2+3\delta_0<\frac 12-\frac
   1{2\alpha}$.
Integrating this proves the lemma, by \eqref{eqEncadrementX} and \eqref{InegNagaev}.
\end{proof}

\begin{lem}\label{lemE_2}
\begin{equation}
\sum_{n\ge 0}{\mathbb P}(E_2(n)^c)<+\infty.
\end{equation}
\end{lem}

\begin{proof}
According to \cite[Thm 1.3]{GKSZ}
applied twice with $u=\gamma_0/2$~: first with the scenery
$(\gamma_0-{\mathbf 1}_{\{a<p_{2y}<b\}})_{y\in\mathbb Z}$
and with the strongly aperiodic Markov chain $(S_{2n}/2)_{n\ge 0}$,
and second with the scenery $(\gamma_0-{\mathbf 1}_{\{a<p_{2y+S_1}<b\}})_{y\in\mathbb Z}$
and with the strongly aperiodic Markov chain $((S_{2n+1}-S_1)/2)_{n\ge 0}$ conditionally on $S_1$, we get
the existence of
$c_1>0$ such that, for every $n\ge 1$, we have
$${\mathbb P}(E_2(n)^c)\le \exp\left(-c_1n^{\frac 13}\right).$$

\end{proof}

\begin{lem}\label{lemE_3}
We have
\begin{equation}
\sum_{n\in\mathbb{N}}\PP( E_4(n)\setminus E_3(n))<\infty.
\end{equation}
\end{lem}
\begin{proof}
We recall that, taken $\omega$ and $S$, $\left(\xi_k-\frac{p_{S_k}}{1-p_{S_k}}\right)_{y,k}$
is a sequence of independent, centered random variables.
For every integer $\nu\ge 2$, there exists a constant $\tilde{C}_\nu>0$ such that
$\left|{\mathbb E}^{\e,\omega}\left[(\xi_{k}-\frac{p_{S_k}}{1-p_{S_k}})^{\nu}|S\right]\right|
\leq \tilde{C}_\nu \left(\frac{1}{1-p_{S_k}}\right)^{\nu}$
$\mathbb P$-almost surely. Consequently, for every $N\ge 1$, there exists a constant $C_N>0$ such that for all
$2\leq \nu\leq 2N$,
$\left|{\mathbb E}^{\e,\omega}\left[(\xi_{k}-\frac{p_{S_k}}{1-p_{S_k}})^{\nu}|S\right]\right|
\leq \left(\frac{C_N}{1-p_{S_k}}\right)^{\nu}$.
Hence, for every $n\ge 1$ and $N\ge 1$, we have on $E_4(n)$~:
\begin{eqnarray*}
{\mathbb E}^{\e,\omega}[(\Vb_{2n}^{2})^{N}\vert S]
 &=&\sum_{k_1=0}^{2n-1}\sum_{k_2=0}^{2n-1}\dots \sum_{k_{N}=0}^{2n-1}
 \EE^{\e,\omega}\left[\prod_{i=1}^N\Xb_{k_i}^2\vert S\right]
\\
 &=&\sum_{k_1=0}^{2n-1}\sum_{k_2=0}^{2n-1}\dots \sum_{k_{N}=0}^{2n-1}
 \EE^{\e,\omega}\left[\prod_{j=0}^{2n-1}\Xb_{j}^{2\theta_j(k_1,\dots k_N)}\vert S\right]
\\
 &\leq&\sum_{k_1=0}^{2n-1}\sum_{k_2=0}^{2n-1}\dots \sum_{k_{N}=0}^{2n-1}
 \prod_{j=0}^{2n-1}\left(\frac{C_N}{1-p_{S_j}}\right)^{2\theta_j(k_1,\dots,k_N)}
\\
   &=&  (C_N)^{2N} \left(\sum_{k=0}^{2n-1}\frac{1}{(1-p_{S_k})^{2}}\right)^{N}
   \leq (C_N)^{2N}n^{d N}
\end{eqnarray*}
where $\theta_j(k_1,k_2,\dots,k_{N}):=\#\{1\leq i\leq N,\ k_i=j\}$.
Consequently, on $E_4(n)$,
$${\mathbb P}^{\e,\omega}(\Vb_{2n}^2> n^{d+\delta_0}\vert S)\leq n^{-(d+\delta_0)N}\EE^{\e,\omega}
\left[(\Vb_{2n}^2)^N\vert S\right]\leq (C_N)^{2N}n^{-\delta_0 N}=O(n^{-2})
$$
by taking $N$ large enough. Integrating this on $E_4(n)$ yields the result.
\end{proof}
\begin{lem}\label{LemmaProbaQueXVautK}
We have on $E_2(n)$, uniformly on $\omega$, $S$  and on $k\in\mathbb Z$:
\begin{equation}
{\mathbb P}^{\e,\omega}\left(X_{2n}=-k\left\vert \, S\right.\right) = O\left(
     \sqrt{\ln(n)n^{-1}}\right).
\end{equation}
\end{lem}
\begin{proof}
On $E_2(n)$, we have~:
\begin{eqnarray*}
{\mathbb P}^{\e,\omega}\left(X_{2n}=-k\vert\, S\right)
    &=& \frac 1{2\pi}
    \int_{-\pi}^\pi{\mathbb E}^{\e,\omega}\left[ e^{it X_{2n}}\vert\, S\right]
      e^{ikt}\, dt\\
  &\le& \frac 1{2\pi}
    \int_{-\pi}^\pi  \left\vert
    {\mathbb E}^{\e,\omega}\left[ e^{it X_{2n}}\vert\, S\right]
      \right\vert\, dt\\
  &\le&\frac 1{\pi}
    \int_{0}^\pi \prod_{y\in P} (\chi_{p_y}(\e_y t))^{N_{2n}(y)}
\end{eqnarray*}
with
$$ \chi_p(t):=\left\vert    {\mathbb E}^{\e,\omega}[e^{i t \xi_0}|p_0=p]
\right\vert=
\left\vert  \frac{1-p}{1-pe^{it}}\right\vert
    = \frac{1-p}{(1+p^2-2p\cos(t))^{\frac{1}2} }.
$$
Since $\chi_p(t)$ is decreasing in $p$ and since $0<a<p_y<b<1$ for $y\in P$,
there exist $0<\beta <\pi/2$ and $c>0$ such that
$$\mbox{for a.e.}\ \omega,\ \
\forall y\in P,\ \ \forall t\in[0,\beta],\ \
\chi_{p_y}(t)\leq \frac{1-a}{(1+a^2-2a\cos(t))^{\frac{1}{2}}}
    \le  \exp(-ct^2).$$
Let us define $a_n:=\sqrt{2\ln(n)/(c \gamma_0 n)}$.
Since $\#\Lambda_n=\sum_{y\in P} N_{2n}(y)\geq \gamma_0 n$ on $E_2(n)$, we have
$$
 \int_{a_n}^{\beta} \prod_{y\in P} (\chi_{p_y}(t))^{N_{2n}(y)}\, dt
\leq  \int _{a_n}^{\beta} \exp(-c t^2 \#\Lambda_n)\ dt
\leq   \int_{a_n}^{\beta} \exp(-c t^2 \gamma_0 n )\ dt
     \leq n^{- 1}
$$
 on $E_2(n)$.
Moreover,
$$ \int_0^{a_n} \prod_{y\in  P} (\chi_{p_y}(t))^{N_{2n}(y)}\, dt\le a_n$$
and
$$ \int_\beta^\pi \prod_{y\in P} (\chi_{p_y}(t))^{N_{2n}(y)}\, dt\le
   \int_\beta^\pi \prod_{y\in P}
         \left(\frac{1-p_y}{1-p_y\cos(\beta) }\right)^{N_{2n}(y)}\, dt
         \le  \pi \left(\frac {1-a}{1-a\cos(\beta)}\right)^{\gamma_0 n/2},$$
since $p>a>0$ for $p\in P$.
\end{proof}

\begin{lem} \label{LemmeProbaAn}
Suppose that
 $\delta':=\delta_3-\frac{\delta_2}2-{\delta_1}>0$ and
$\delta_3+\delta_2<\frac 14$.
Then, uniformly on $p_0$ and $(S_k)_k$,
$${\mathbb P}(A_n\setminus B_n \vert S,p_0) =O\left( n^{-\delta'}\right).$$
\end{lem}
\begin{proof}
Up to an enlargement of the probability space, we consider a centered gaussian random variable
$G$ with variance $n^{ \frac 32-2\delta_3}$ independent of $(\omega,S)$.
We have
$${\mathbb P}(|D_{2n}| \le n^{\frac 34-\delta_3}\vert S,p_0)
{\mathbb P}(|G|\le n^{\frac 34-\delta_3})
   \le {\mathbb P}\left(|D_{2n}+G|\le 2n^{\frac 34-\delta_3}
\vert S,p_0 \right) $$
and so
$${\mathbb P}(|D_{2n}| \le n^{\frac 34-\delta_3}\vert S,p_0) \le
  {\mathbb P}\left(|D_{2n}+G|\le 2n^{\frac 34-\delta_3} \vert S,p_0
\right) /0,6.$$
Let $\tilde\chi$ be the characteristic function of $\frac {p_0}{1-p_0}$.
Since $p_0$ is non-constant, there exist $\tilde\beta>0$ and $\tilde{c}>0$ such that
\footnote{Applying \cite[Lemma 3.7.5, p. 58]{Lukacs} to the random variable $Y:=
\frac{p_0}{1-p_0}-\frac{p_1}{1-p_1}$ which is not identically equal to 0
and whose characteristic function is
$|\tilde\chi|^2$, we get that for every $r>0$ and every $t\in\left[-\frac 1r;\frac 1r\right]$,
$|1-|\tilde\chi(t)|^2|\ge \frac{t^2}3 {\mathbb E}[Y^2{\mathbf 1}_{\{|Y|\le r\}}]$.
We take $\tilde\beta$ such that
$\tilde{c}:= \frac 16 {\mathbb E}[Y^2{\mathbf 1}_{\{|Y|\le \tilde\beta^{-1}\}}]>0$.
For every $u\in[- \tilde\beta;\tilde\beta]$, we have
$|1-|\tilde\chi(u)||\ge\frac 12 |1-|\tilde\chi(u)|^2|\ge  {\tilde{c}} u^2$  and so
$|\tilde\chi(u)|\le 1- {\tilde{c}} u^2$.}
$$\forall u\in[- \tilde\beta;\tilde\beta],\ \  \vert\tilde\chi(u)\vert\le e^{-\tilde{c}u^2}.$$
Consequently,

$\displaystyle {\mathbb P}\left(|D_{2n}+G|\le 2n^{\frac 34-\delta_3}\vert S,p_0
\right)
=  \frac{2n^{\frac 34-\delta_3}}\pi\int_{\mathbb R}
  \frac{\sin(2tn^{\frac 34-\delta_3})}
   { 2tn^{\frac 34-\delta_3}}
{\mathbb E}[e^{itD_{2n}}\vert S,p_0]{\mathbb E}[e^{itG}]\, dt$
\begin{eqnarray*}
&=&  \frac{2n^{\frac 34-\delta_3}}\pi\int_{\mathbb R}
   \frac{\sin(2tn^{\frac 34-\delta_3})}
   { 2tn^{\frac 34-\delta_3}}
e^{it\frac{p_0}{1-p_0}N_{2n}(0)}\prod_{y\ne 0} \tilde\chi
\left(\e_y N_{2n}(y)t\right)e^{-\frac{t^2}2n^{\frac 32-2\delta_3}}\, dt\\
&\le&\frac{2n^{\frac 34-\delta_3}}\pi \int_{\mathbb R}
\prod_{y\ne 0} \left\vert\tilde\chi
\left( \e_y N_{2n}(y)t\right)\right\vert e^{-\frac{t^2}2n^{\frac 32-2\delta_3}}\, dt.
\end{eqnarray*}
Let $\delta_4>0$ be such that $\delta_5:=\frac 14-\delta_3-\delta_2-\delta_4>0$
and let $b_n:=n^{\delta_4+\delta_3-\frac 34 }$.
On the one hand, we have
\begin{eqnarray*}
I_1:=\int_{\{|t|>\tilde\beta b_n\}}
\prod_{y\ne 0} \left\vert\tilde\chi
\left(  \e_y N_{2n}(y)t\right)\right\vert e^{-\frac{t^2}2n^{\frac 32-2\delta_3}}\, dt
&\le& \int_{\{|t|>\tilde\beta b_n\}} e^{-\frac{t^2}2n^{\frac 32-2\delta_3}}\, dt\\
&\le& n^{\delta_3-\frac 34 }\int_{\{|s|>\tilde\beta n^{\delta_4}\}} e^{-s^2/2}\, ds\\
&\le& 2 n^{\delta_3-\frac 34 } e^{-\tilde\beta^2n^{2\delta_4}/2}.
\end{eqnarray*}
On the other hand, we will estimate the following quantity on $A_n$~:

$\displaystyle I_2:=\int_{\{|t|\le \tilde\beta b_n\}}
\prod_{y\ne 0} \left\vert\tilde\chi
\left( \e_y N_{2n}(y)t\right)\right\vert e^{-\frac{t^2}2n^{\frac 32-2\delta_3}}\, dt.$

Let us define $F_n:=\{y\ne 0\, :\, N_{2n}(y)\ge n^{1/2-\delta_1}/2\}$
and $\rho_n:=\# F_n$.
On $A_n$, we have $2n-n^{1/2+\delta_2}\le\sum_{y\ne 0}N_{2n}(y)\le
    \rho_nn^{1/2+\delta_2}+(2n^{1/2+\delta_1}-\rho_n)\frac{n^{1/2-\delta_1}}2$
and hence $\rho_n\ge n^{1/2-\delta_2}/2$ (if $n$ is large enough). Therefore, on $A_n$, we have
$\alpha_n:=\sum_{y\in F_n}N_{2n}(y)\ge n^{1-\delta_2-\delta_1}/4$.
Now, using the H\"older inequality, we have
\begin{eqnarray*}
I_2&\le& \prod_{y\in F_n} \left(\int_{\{|t|\le \tilde\beta b_n\}}
 \left\vert\tilde\chi
\left( \e_y N_{2n}(y)t\right)\right\vert^{\frac{\alpha_n}{N_{2n}(y)}} \, dt\right)
       ^{\frac{N_{2n}(y)}{\alpha_n}}\\
&\le&\sup_{y\in F_n}
\left(\int_{\{|t|\le \tilde\beta b_n\}}
 \left\vert\tilde\chi
\left(\e_y N_{2n}(y)t\right)\right\vert^{\frac{\alpha_n}{N_{2n}(y)}} \, dt\right)\\
&\le&b_n \sup_{y\in F_n}
\left(\int_{|v|\le \tilde\beta}
 \left\vert\tilde\chi
\left(\e_y  N_{2n}(y)vb_n\right)\right\vert^{\frac{\alpha_n}{N_{2n}(y)}} \, dv\right).
\end{eqnarray*}
Let us notice that, if $|v|\le \tilde\beta$,  we have on $A_n$,
$$|\e_y N_{2n}(y)vb_n |\le \tilde\beta n^{1/2+\delta_2}n^{\delta_4+\delta_3-\frac 34 }
   =\tilde\beta  n^{-\delta_5}\le\tilde\beta,$$
since $\delta_5>0$.
Hence, on $A_n$, we have
\begin{eqnarray*}
I_2
&\le&b_n  \sup_{y\in F_n}
\left(\int_{\{|v|\le \tilde\beta\}}e^{-\tilde{c} (N_{2n}(y))^2v^2 n^{2\delta_4+2\delta_3-
     \frac 32 }
\frac{\alpha_n}{N_{2n}(y)}} \, dv\right)\\
&\le&b_n \sup_{y\in F_n}
\left(\int_{\{|v|\le \tilde\beta\}}e^{-\tilde{c} N_{2n}(y)v^2 n^{2\delta_4+2\delta_3
-\frac 12 -\delta_2-\delta_1}
        /4}\, dv\right)\\
&\le& \sup_{y\in F_n} \frac {b_nn^{-\delta_3-\delta_4+\frac{\delta_2+\delta_1}2
     +\frac 14 }}
     {\sqrt{N_{2n}(y)}}
\left(\int_{\mathbb R}e^{-\tilde{c} s^2/4}\, ds\right)\\
&\le& \sqrt{2}n^{-\frac 3 4 +\delta_1+\frac{\delta_2}2} \int_{\mathbb R}e^{-\tilde{c} s^2/4}\, ds.
\end{eqnarray*}
Hence, uniformly on $A_n$ and on $p_0$, we have
$${\mathbb P}(A_n\setminus B_n \vert (S_k)_k,p_0) =
  O(n^{\delta_1+\frac{\delta_2}2-\delta_3}).$$
\end{proof}
\begin{lem}\label{lemB_n}
Under the same hypotheses, we have
$$\sum_n{\mathbb P}(S_{2n}=0,\, X_{2n}\le 0\le X_{2n+1};A_n\cap E_2(n)\setminus B_n)<\infty.$$
\end{lem}
\begin{proof}
According to Lemma \ref{LemmaProbaQueXVautK}, Lemma \ref{LemmeProbaAn} and since
${\mathbb P}(S_{2n}=0)=O(n^{-1/2})$ and $\EE[1/(1-p_0)]<\infty$,
we have

$\displaystyle \EE\left[\sum_{k=0}^\infty \PP^{\e,\omega}(S_{2n}=0,
X_{2n}=-k, A_n\cap E_2(n)\setminus B_n)p_0^k\right]
$
\begin{eqnarray}
&=& \EE\left[\sum_{k=0}^\infty p_0^k \mathbf{1}_{\{S_{2n}=0\}}
\mathbf{1}_{(A_n\setminus B_n)\cap E_2(n)}\PP^{\e,\omega}(X_{2n}=-k\vert S)\right]\nonumber\\
&\leq & C \sqrt{(\ln n)n^{-1}}\EE\left[\frac 1{1-p_0} \mathbf{1}_{\{S_{2n}=0\}}
      \PP(A_n\setminus B_n\vert S,p_0)\right]\nonumber\\
& = & O(n^{- 1-\delta'}\sqrt{\ln n})\label{InegProbaSerieAB}.
\end{eqnarray}
\end{proof}
\begin{lem}\label{lemE_4}
If $\delta_0\alpha<\delta_1$, we have
\begin{equation}
\sum_{n}\PP\left(S_{2n}=0, X_{2n}\le 0 \le X_{2n+1}, E_4(n)^c, A_n, E_2(n), E_0(n)\right)
<+\infty.
\end{equation}
\end{lem}
\begin{proof}
%Let $d:=\frac{1}{\alpha}+\frac{1}{2} +\frac{2}{\alpha}(\delta_1+\delta_2/2)+c\delta_0$.
We notice that on $E_0(n)\cap A_n$,
%we have for $p>2/\alpha$ but such that $\frac{1}{2}-\frac{1}{\alpha p}<\delta_0$ :

$\displaystyle \PP(E_4(n)^c \vert S,p_0)$
\begin{eqnarray}
& \leq &
   n^{-d\alpha/2}\EE\left[\left(\sum_{y=-n^{1/2+\delta_1}}^{n^{1/2+\delta_1}}\frac{1}
   {(1-p_y)^2}N_{2n}(y)\right)^{\alpha/2}\Big\vert S,p_0\right]\nonumber\\
&\leq&
   n^{-d\alpha/2}\EE\left[\sum_{|y|\leq n^{1/2+\delta_1}, y\neq 0}\frac{1}{(1-p_y)^{\alpha}}
   N_{2n}^{\alpha/2}(y)
   +\frac{1}{(1-p_0)^{\alpha}}N_{2n}^{\alpha/2}(0)\Big\vert S,p_0\right]\label{InegE4a}\\
& \leq & n^{-d\alpha/2} \left(2n^{1/2+\delta_1}\EE\left[\frac{1}{(1-p_0)^{\alpha}}\right]
+n^{\frac 12+\delta_0\alpha}\right)n^{(1/2+\delta_2)\alpha/2}=O(n^{-3\delta_0\alpha/2}),
\label{InegE4}
\end{eqnarray}
since $\alpha\leq 2$, $\delta_0\alpha<\delta_1$ and $d=\frac 12+\frac 1\alpha+\frac{2\delta_1}\alpha
 +3\delta_0+\delta_2$.
Similarly as in \eqref{InegProbaSerieAB}, this yields
\begin{equation*}
\EE\left[\sum_{k=0}^\infty \PP^{\e,\omega}(S_{2n}=0, X_{2n}=-k, E_4(n)^c\cap A_n
\cap E_2(n)\cap E_0(n))p_0^k\right]
=O(n^{-1-3\delta_0\alpha/2}\sqrt{\ln n}).
\end{equation*}
Hence we have
\begin{equation*}
\PP\left(S_{2n}=0, X_{2n}\le 0 \le X_{2n+1}, E_4(n)^c, A_n, E_2(n), E_0(n)\right)
=O(n^{-1-\delta_0\alpha/2}\sqrt{\ln n}).
\end{equation*}
\end{proof}
\begin{lem}\label{lemE_1} If $\delta_0<\frac{1}{2}(1-\frac{1}{\alpha})$, we have
$$\sum_n{\mathbb P}\left(S_{2n}=0,\, X_{2n}\le 0\le X_{2n+1},\, E_2(n)\setminus E_1(n)\right)
    <\infty. $$
\end{lem}
\begin{proof}
Notice that on $\{\frac{1}{1-p_0}\leq c_n\}$, we have
$$\PP(E_1(n)^c\vert p_0)\leq 2n^{1/2+\delta_1}\PP\left(\frac{1}{1-p_0}> c_n\right)\leq
   \frac{2n^{1/2+\delta_1}}{c_n^{\alpha}}\EE\left[\left(\frac{1}{1-p_0}\right)^{\alpha}\right]
=O(n^{-\delta_0\alpha}) .$$
Similarly as in \eqref{InegProbaSerieAB}, since ${\mathbb E} [1/(1-p_0)]<\infty$,
for $\delta_0$ small enough, we have

$\displaystyle
 \EE\left[\sum_{k=0}^\infty \PP^{\e,\omega}(S_{2n}=0,
 X_{2n}=-k, E_2(n)\setminus E_1(n))p_0^k\right]$
\begin{eqnarray*}
&=& \EE\left[\sum_{k=0}^\infty p_0^k \mathbf{1}_{\{S_{2n}=0\}}
   \mathbf{1}_{E_2(n)\setminus E_1(n)}\PP^{\e,\omega}(X_{2n}=-k\vert S)\right]\\
&\leq & C \sqrt{(\ln n)n^{-1}}
 n^{-1/2}\left(\EE\left[\sum_{k=0}^\infty   \PP(E_1(n)^c\vert p_0)
   \mathbf{1}_{\{(1-p_0)^{-1}\leq c_n\}}p_0^k\right]
+\EE\left[\frac{1}{1-p_0}\mathbf{1}_{\{(1-p_0)^{-1}> c_n\}}\right] \right)\\
& = & O(n^{-1-c\delta_0}\sqrt{\ln n}),
\end{eqnarray*}
where we can use H\"{o}lder's inequality,  to deal with the second term of the third line, since $\alpha>1$ and $\delta_0<\frac{1}{2}(1-\frac{1}{\alpha})$.
\end{proof}

We take $\delta_3\in\left(0,\frac 12-\frac 1{2\alpha}\right)$ (since $\alpha>1$)
and then $\delta_1>0$ and $\delta_2>0$ such that
$$\delta_1<\frac 16,\ \delta_2<\frac 16,\ \delta_2<\frac 14-\delta_3,\
  \frac{\delta_1}\alpha+\frac{\delta_2}2<\frac 12-\frac 1{2\alpha}-\delta_3,\
  \delta_1+\frac{\delta_2}2<\delta_3 $$
and finally $\delta_0$ such that
$$\delta_0<\frac 18,\ \delta_0\alpha<\delta_1\ \mbox{and}\
\  \frac{\delta_1}\alpha+\frac{\delta_2}2+3\delta_0<\frac 12-\frac 1{2\alpha}-\delta_3.$$
%\mbox{and}\ \frac {\delta_1}\alpha+3\delta_0<\frac 34-\frac 1{2\alpha}. $$
Combining all the previous lemmas with these choices
for $\delta_0,\delta_1,\delta_2,\delta_3$, we get \eqref{eqSommePourTheoremeTransience},
which proves Theorem \ref{sec:thm1}.

\section{Proof of Theorem \ref{sec:TL} (functional limit theorem)}
We assume that $(p_k)_k$ satisfies the conditions of Theorem \ref{sec:TL}.
\begin{lem}\label{cvps}
Let $(\e_k)_k$ be a (fixed or random) sequence with values in $\{-1;1\}$.
Let $(p_k)_k$ be as in Theorem \ref{sec:TL}.
Then, under $\PP$, the sequence of random variables
$$\left(\left(n^{-\delta}\sum_{k=0}^{\lfloor n t\rfloor-1}\e_{S_k}\left(\xi_k- \frac{p_{S_k}}
    {1- p_{S_k}} \right), 0\right)_{t\geq 0}\right)_n $$
converges in distribution (in the space of Skorokhod ${\mathcal D}([0;+\infty),{\mathbb
  R}^2)$) to $(0,0)_{t\geq 0}$.
\end{lem}
\begin{proof}
We first notice that it is enough to prove that
$$N^{-\delta}\sup_{0\leq n\leq N}\left|\sum_{k=0}^{n-1}\e_{S_k}\left(\xi_k- \frac{p_{S_k}}
    {1- p_{S_k}} \right)\right|\to_{N\to +\infty}0$$
in probability.

Let us define
$$\tilde E_4(N,v):= \left\{ \sum_{k=0}^{N-1}\frac 1{(1-p_{S_k})^2}\le N^{\frac 12
    +\frac 1\beta +v} \right\}.$$

%We introduce ${\tilde{\mathcal A}}_N:=\{\max_{k\leq N}|S_k|\leq N^{1/2+\delta_0}, \max_{y\in\Z}|N_N(y)|\leq %N^{1/2+\delta_0}\}$.
We proceed as in formula (\ref{InegE4a}) (with a conditioning with respect to $S$ only,
%$\delta_0$ well chosen
and $\alpha<\beta$ but close enough to $\beta$)  to prove that ${\mathbb P}[(\tilde E_4(N,v))^c | S] \leq  N^{-c v }$ on
$A_N$ for $c>0$ and $N$ large enough. Moreover, ${\mathbb P}( A_N^c)\to_{N\to+\infty} 0$ by Lemma \ref{lemA_n}, which gives
%to the proof of Lemma \ref{lemA_n}
%(with a conditioning with respect to $S$ only),
%we know that, for every $v>0$ small enough (and so for every $v>0$), we have
$$\lim_{N\rightarrow +\infty}{\mathbb P}\left(\tilde E_4(N,v)\right) =1.$$
Now, taken $(\e, S,\omega)$, $\left(\sum_{k=0}^{n-1}\e_{S_k}\left(\xi_k-\frac{p_{S_k}}{1-p_{S_k}}
\right)
\right)_n$
is a martingale. Hence, according to the maximal inequality for martingales we have, for every $\theta>0$,
\begin{eqnarray*}
{\mathbb P}^{\e,\omega}\left(\sup\limits_{n\le N}\left|\left. \left(
\sum_{k=0}^{n-1}\e_{S_k}\left(\xi_k-\frac{p_{S_k}}{1-p_{S_k}}\right)\right)^2
   \right|\ge
     \theta^2 N^{2\delta}
  \right\vert S\right)
  &\le&\frac{ 2 \sup\limits_{n\le N}  {\mathbb E}^{\e,\omega}\left[\sum_{k=0}^{n-1}
     \left(\xi_k-\frac{p_{S_k}}{1-p_{S_k}}
       \right)^2|S\right]}{\theta^2N^{2\delta}}\\
  &\le& \frac{ 2 \sum_{k=0}^{N-1}\frac 1{(1-p_{S_k})^2}}  {\theta^2N^{2\delta}}\\
  &\le& \frac{ 2  N^{\frac 12  +\frac 1\beta +v}  }  {\theta^2N^{2\delta}}
     =2 N^{-\frac 12+v}\theta^{-2},
\end{eqnarray*}
on $\tilde E_4(n,v)$, since $\delta=\frac{1}{2}+\frac{1}{2\beta}$.
Hence, we get
$${\mathbb P}\left(\sup_{n\le N}\left|
\sum_{k=0}^{n-1}\e_{S_k}\left(\xi_k-\frac{p_{S_k}}{1-p_{S_k}}\right)\right|\ge \theta N^\delta
  \right)\le 1-{\mathbb P}(\tilde E_4(N,v))+ 2 N^{-\frac 12+v}\theta^{-2}.
$$
From this we conclude that
$\lim_{n\rightarrow+\infty }{\mathbb P}\left(\sup_{n\le N}\left|
\sum_{k=0}^{n-1}\e_{S_k}\left(\xi_k-\frac{p_{S_k}}{1-p_{S_k}}\right)\right|\ge
   \theta N^\delta
  \right)=0$.
\end{proof}
The next lemma follows from the proof of \cite[Thm 4]{GLN2} when $\beta=2$.
The proof of the general case $\beta\in(1,2]$ is postponed to Section
\ref{SectionPreuvDuLemmeCV}.
\begin{lem}\label{LemmaCVdansDenDimension2}
Let $\beta\in(1,2]$.
Let $S=(S_n)_{n\ge 0}$ be a random walk on $\mathbb Z$ starting from $S_0=0$, with
iid
centered square integrable and non-constant increments and such that
$gcd\{k\, :\, {\mathbb P}(S_1=k)>0\}=1$.
Let $(\tilde\e_y)_{y\in\mathbb Z}$ be a sequence of iid random variables independent of $S$
with symmetric distribution and such that
$(n^{-\frac 1\beta}\sum_{k=1}^n\tilde\e_k)_n $ converges in distribution to a random variable
$Y$ with stable distribution of index $\beta$.
Then, the following convergence holds in distribution in $\mathcal{D}([0,+\infty),\RR^2)$
\begin{equation*}
\left(n^{-\delta}\sum_{k=0}^{\lfloor nt\rfloor-1}\tilde\e_{S_k},
   n^{-\frac 12}S_{\lfloor n t \rfloor}\right)_{t\geq 0}\longrightarrow_{n\to+\infty}
    (\tilde\Delta_t, \tilde B_t)_{t\geq 0},
\end{equation*}
with $\delta=\frac 12+\frac 1{2\beta}$,
where $( \tilde B_t)_t$ is a Brownian motion such that $Var(\tilde B_1)=Var(S_1)$ and with
$(\tilde L_t(x))_{t,x}$ the jointly continuous version
of its local time and where
$$\tilde\Delta_t:=\int_{\mathbb R}\tilde L_t(x)\, d\tilde Z_x ,$$
with $\tilde Z$ independent of $\tilde B$ given by two independent right continuous
stable processes $(\tilde Z_x)_{x\ge 0}$ and $(\tilde Z_{-x})_{x\ge 0}$
with stationary independent increments
such that $\tilde Z_1$, $\tilde Z_{-1}$
have the same distribution as $Y$.
\end{lem}

%\begin{proof}
%Let us prove the convergence taken $(\omega,(S_k)_k)$, in the case $\beta=2$.
%Taken $(\omega,(S_k)_k)$, the random variables $\xi_k$'s are mutually independent,
%with expectations $\frac{p_{S_k}} {1- p_{S_k}} $ and variances
% $\frac{p_{S_k}} {(1- p_{S_k})^2} $.
%Hence, according to the Kolmogorov strong law of large number, it is enough
%to see that $\sum_{k\ge 0}\frac{Var^\omega(\xi_{S_k}^{(k)} |(S_p)_p) } {(k+1)^{3/2}}<+\infty$.
%We have
%$$ \sum_{k\ge 0}\frac{Var^\omega(\xi_{S_k}^{(k)} |(S_p)_p) } {(k+1)^{3/2}}
%       =\sum_{k\ge 0}\frac {p_{S_k}}{(k+1)^{3/2}(1-p_{S_k})^2}.$$
%This quantity is almost surely finite since its expectation is
%equal to ${\mathbb E}[p_0(1-p_0)^{-2}]\sum_{k\ge 1}k^{-3/2}$
%which is finite.
%\end{proof}
Now, we prove a functional limit theorem for $(X_{\floor{nt}},S_{\floor{nt}})$
from which we will deduce our theorem \ref{sec:TL}.
\begin{prop}\label{sec:TLF}
Under the assumptions and with the notations of Theorem \ref{sec:TL}, the sequence of processes
$$\left(\left(n^{-\delta}X_{\floor {nt}},n^{-1/2}S_{\floor {nt}}\right)_{t\ge 0}\right)_n$$
converges in distribution under $\PP$ (in the space of Skorokhod ${\mathcal D}([0;+\infty),{\mathbb
  R}^2)$) to
$(\sigma\Delta_t,B_t)_{t\ge 0}$.
\end{prop}

\begin{proof}[Proof of Proposition \ref{sec:TLF}]
We observe that $X_n$ can be rewritten
$$X_n= \sum_{k=0}^{n-1}\e_{S_k}\left(\xi_{k}-\frac{p_{S_k}}
     {1-p_{S_k}}\right) + \sum_{k=0}^{n-1}\e_{S_k}\frac{p_{S_k}}{1-p_{S_k}} .$$

%In the case $\beta=2$,
%let us write $E_n:=\sum_{k=0}^{n-1}\e_{S_k}\left(\xi_{S_k}^{(k)}-\frac{p_{S_k}}
%     {1-p_{S_k}}\right)$.
%According to Lemma \ref{cvps},  the sequence of processes
%$\left(\left(n^{-3/4} E_{\floor{nt}}   \right)_{t\ge 0}\right)_n$
%converges almost surely (in  ${\mathcal D}([0;+\infty),{\mathbb
%  R})$) to 0.
%Indeed, let $T>0$,
%on the event $\{n^{-3/4}E_n\rightarrow 0\}$, for every $m\ge 1$, we have
%$$\sup_{0\le t\le m/n} \left| \frac{E_{\floor{nt}}} {n^{3/4}}\right|
%   \le  \frac{\sup_{k\le m} |E_k|} {n^{3/4}}
%\ \ \mbox{and}\ \
%\sup_{(m/n)\le t\le T} \left| \frac{E_{\floor{nt}}} {n^{3/4}}\right|
%    \le   \sup_{k\ge m}\left|\frac{E_{k}} {k^{3/4}}\right| T^{3/4}.$$
%We conclude by taking $m$ large enough such that
%$\sup_{k\ge m} \frac{|E_{\floor k}|} {k^{3/4}} < \e /(2T^{3/4})$
%and then $n_0$ large enough such that
%$\frac{\sup_{k\le m} E_k} {n^{3/4}} <\e/2$ for every $n\ge n_0$.

\noindent
According to Lemma \ref{cvps},
it is enough to prove, under $\PP$,  the convergence
\begin{equation}\label{eqConvergenceP14}
\left(\left(n^{-\delta} \sum_{k=0}^{\floor{nt}-1} \e_{S_k}\frac{p_{S_k}}{1-p_{S_k}}
     ,n^{-1/2}S_{\floor {nt}}\right)_{t\ge 0}\right)_n\to_{n\to+\infty} (\sigma\Delta_t,B_t)_{t\ge 0}
\end{equation}
in distribution in ${\mathcal D}([0;+\infty),{\mathbb R}^2)$.

In case {\it (b)}, $\left(\tilde\e_y:=\e_y\frac{p_{y}}{1-p_{y}}\right)_y  $
is a sequence of independent identically distributed random variables with symmetric distribution
such that
$(n^{-1/\beta}\sum_{y=1}^n\tilde\e_y)_n$ converges in distribution to a random variable
with characteristic function $\theta\mapsto \exp(-A_1|\theta|^\beta)$, where $A_1:=\mathbb{E}(p_0^2/(1-p_0)^2)/2$ if $\beta=2$.
Hence the result follows from Lemma \ref{LemmaCVdansDenDimension2}.

In case {\it (a)} with $\beta=2$, we observe that $\sum_{k=0}^{n-1}\e_{S_k}$ is equal to 0
if $n$ is even and is equal to 1 if $n$ is odd. Hence,
$((n^{-3/4}\sum_{k=0}^{\floor{nt}-1}\e_{S_k})_{t\geq 0})_n$ converges to 0 in
${\mathcal D}([0;+\infty),{\mathbb
  R})$ and it remains to prove the convergence of
$$\left(\left(n^{-3/4} \sum_{k=0}^{\floor{nt}-1} \e_{S_k}\left(\frac{p_{S_k}}{1-p_{S_k}}
       -{\mathbb E}\left[\frac{p_0}{1-p_0}\right]\right)
     ,n^{-1/2}S_{\floor {nt}}\right)_{t\ge 0}\right)_n.$$
Let us write $\lambda$ for the characteristic function of $\frac{p_0}{1-p_0}-
    {\mathbb E}\left[\frac{p_0}{1-p_0}\right]$.
Since $\frac{p_0}{1-p_0}$ has a finite variance and $\lambda(\e_y\cdot)$ behaves as $\lambda$ at 0, we can follow
the proof of the convergence of the finite distributions of \cite[prop 1]{GLN2},
which gives the convergence in distribution in ${\mathcal D}([0;+\infty),{\mathbb
  R}^2)$ thanks to the tightness that can be proved for the first coordinate
as in \cite{KestenSpitzer}.

Now, let us explain how case {\it (a)} with $\beta\in(1,2)$ will also be deduced from
Lemma \ref{LemmaCVdansDenDimension2}. This comes from the following lemma.
\begin{lem}\label{souslem}
Let $\beta\in(1,2)$.
Let $S=(S_n)_n$ be a simple symmetric random walk on $\mathbb Z$ starting from $S_0=0$.
Let $(\tilde a_y)_{y\in\mathbb Z}$ be a sequence of iid random variables such that $\mathbb{E}(|\tilde{a}_0|)<\infty$,
independent of $S$.
We have
$$\left(n^{-\delta}\left(\sum_{k=0}^{\lfloor nt\rfloor-1}
    (-1)^k\tilde a_{S_k}-\sum_y
   \left(\tilde a_{2y}-\tilde a_{2y-1}\right)
    N_{\lfloor nt\rfloor}(2y) \right),0\right)_{t\ge 0}
    %\stackrel{\mathbb P}
    \longrightarrow (0,0) $$
in distribution as $n$ goes to infinity (in ${\mathcal D}([0;+\infty),{\mathbb
  R}^2)$), with $\delta:=\frac 12+\frac 1{2\delta}$.
\end{lem}
\begin{proof}[Proof of Lemma \ref{souslem}]
Let us write
$$
e_n:=\sum_{k=0}^{n-1} (-1)^k\tilde a_{S_k}-\sum_y
   \left(\tilde a_{2y}-\tilde a_{2y-1}\right)
    N_n(2y).$$
We notice that it is enough to prove that
$$n^{-\delta}\sup_{0\le k\le n}\left|e_k\right|\stackrel{\mathbb P}
   {\longrightarrow}_{n\rightarrow  +\infty}0 .$$
Let $\eta>0$ be such that $2\eta<\frac 1{2\beta}-\frac 14$ (such a $\eta$ exists
since $\beta<2$). For every $n\ge 1$, we consider the set $\Omega'_n$
defined by
$$\Omega'_n:=\left\{\sup_{k\le n}|S_k|\le n^{\frac 12+\eta},\ \sup_{0\le k\le n}
     \sup_{|y|\le n^{\frac 12+\eta}+1}|N_k(y)-N_k(y-1)|\le n^{\frac 14
         +\eta}\right\}. $$
Let us show that $\lim_{n\rightarrow+\infty}{\mathbb P}
  (\Omega'_n)=1$.
As in Lemma \ref{lemA_n}, we have,
$$\lim_{n\rightarrow+\infty}
{\mathbb P}\left(\sup_{k\le n}|S_k|\le n^{\frac 12+\eta}\right)=1.$$
Now we recall that for any even integer $m$,
$$ \sup_y{\mathbb E}[|N_n(y)-N_n(y-1)|^m]=O(n^{\frac m 4}),$$
as $n$ goes to infinity
(see \cite[lem 3]{KestenSpitzer} and \cite[p. 77]{JP}).
Hence, using again the Markov inequality and taking $m$ large enough, we get
$$
  {\mathbb P} (\Omega'_n) \ge 1-o(1)-3n^{\frac 32+\eta}\sup_{n,y}
          \frac{{\mathbb E}[|N_n(y)-N_n(y-1)|^m]}{n^{\frac m4+\eta m}}=1-o(1).
$$
On $\Omega'_n$, using the fact that
$$ \sum_{\ell=0}^{k-1} (-1)^\ell\tilde a_{S_\ell}=\sum_y\left(
     \tilde a_{2y}N_k(2y)-\tilde a_{2y-1}N_k(2y-1)\right),$$
for every $k=0,...,n$, we have
$$|e_k|=\left\vert\sum_y \tilde a_{2y-1} (N_k(2y)-N_k(2y-1))\right\vert
\le \sum_{|y|\le n^{\frac 12+\eta}+1}|\tilde a_y|n^{\frac 14+\eta}.$$
Hence, thanks to the Markov inequality, we get for $\theta>0$.
\begin{eqnarray*}
{\mathbb P}\left(n^{-\delta}\sup_{0\le k\le n}|e_k|>\theta\right) &\le&
(1-{\mathbb P}(\Omega'_n))+{\mathbb P}\left(\sum_{|y|\le n^{\frac 12+\eta}+1}|\tilde a_y|
         >\theta n^{\frac 12+\frac 1{2\beta}-\frac 14-\eta}\right)\\
&\le& (1-{\mathbb P}(\Omega'_n)) + \frac{3{\mathbb E}\left(|\tilde{a}_0|\right)}{
  \theta} n^{-\frac 1{2\beta}+\frac 14+2\eta}.
\end{eqnarray*}
Hence, for every $\theta>0$, we have $\lim_{n\rightarrow+\infty}{\mathbb P}(n^{-\delta}\sup_{0\le k\le n}|e_k|>
\theta)=0$.
\end{proof}
Now we observe that the characteristic function of $\tilde\e_y:= \frac {p_{2y}}{1-p_{2y}}
   -\frac{p_{2y-1}}
   {1-p_{2y-1}}$ is $t\mapsto \left|\tilde\chi\left( t\right)\right|^2$
(where $\tilde\chi$ stands for the characteristic function of $\frac {p_0}{1-p_0}$).
The distribution of $\tilde\e_0$ is symmetric and $(n^{-\frac 1\beta}\sum_{k=1}^n
\tilde\e_k)_n$ converges in distribution to a random variable with characteristic function
$\theta\mapsto \exp(-2A_1|\theta|^\beta)$.
According to Lemma \ref{LemmaCVdansDenDimension2} applied with the random walk $\left(\tilde
S_{k}:=\frac{S_{2k}}2\right)_k$, we have
\begin{equation*}
\left(n^{-\delta}\sum_{k=0}^{\lfloor nt\rfloor-1}\tilde\e_{\tilde S_k},
   n^{-\frac 12}\frac{S_{\lfloor 2n t \rfloor}}2\right)_{t\geq 0}
  \longrightarrow_{n\to+\infty}
    (\tilde\Delta_t,  \tilde B_t)_{t\geq 0},
\end{equation*}
in distribution in ${\mathcal D}([0;+\infty),{\mathbb R}^2)$,
where $( \tilde B_t)_t$ is a Brownian motion such that $Var(\tilde B_1)=\frac 12$ and with
$(\tilde L_t(x))_{t,x}$ the jointly continuous version
of its local time and where
$$\tilde\Delta_t:=\int_{\mathbb R}\tilde L_t(x)\, d\tilde Z_x ,$$
with $\tilde Z$ independent of $\tilde B$ given by two independent right continuous
stable processes $(\tilde Z_x)_{x\ge 0}$ and $(\tilde Z_{-x})_{x\ge 0}$, the
characteristic functions of $\tilde Z_1$ and of $\tilde Z_{-1}$ being
$\theta\mapsto \exp(-2A_1|\theta|^\beta)$.
Hence, we have
\begin{equation*}
\left(n^{-\delta}\sum_{k=0}^{\lfloor nt/2\rfloor-1}\tilde\e_{\tilde S_k},
   n^{-\frac 12}\frac{S_{\lfloor n t \rfloor}}2\right)_{t\geq 0}
  \longrightarrow_{n\to+\infty}
    (\tilde\Delta_{t/2},  \tilde B_{t/2})_{t\geq 0},
\end{equation*}
and so
\begin{equation*}
\left(n^{-\delta}\sum_{y}\tilde\e_{y}N_{\lfloor nt\rfloor}(2y),
   n^{-\frac 12} S_{\lfloor n t \rfloor} \right)_{t\geq 0}
  \longrightarrow_{n\to+\infty}
    (\tilde\Delta_{t/2},  B_{t})_{t\geq 0},
\end{equation*}
with $B_t:=2\tilde B_{t/2}$. Now we observe that
$$\tilde\Delta_{t/2}=\int_{\mathbb R}\tilde L_{t/2}(x)\, d\tilde Z_x
    =\int_{\mathbb R} L_t(2x)\, d\tilde Z_x=\int_{\mathbb R} L_t(x)\, d Z_x,$$
where $L$ denotes the local time of $B$ and with $Z_x:=\tilde Z_{x/2}$.
Now Lemma \ref{souslem} applied to $\left(\frac{p_y}{1-p_y}\right)_{y\in \mathbb Z}$
gives \eqref{eqConvergenceP14}, which proves
Proposition \ref{sec:TLF} in the case (a) with $\beta\in(1,2)$.
\end{proof}

\begin{proof}[Proof of Theorem \ref{sec:TL}]
We recall that for every $n$, we have
$$X_n = M_{T_n}^{(1)} \ \ \mbox{and}\ \ S_n=M_{T_n}^{(2)}. $$
Moreover we observe that we have
$$T_n=\sum_{k=0}^{n-1}\left(\xi_{k}+1\right), $$
that can be rewritten
$$T_n= \sum_{k=0}^{n-1}\left(\xi_k- \frac{p_{S_k}}{1-p_{S_k}} \right)
    + \sum_{k=0}^{n-1}\left( \frac{p_{S_k}}{1-p_{S_k}}
       - {\mathbb E}\left[\frac{p_0}{1-p_0}\right] \right)
    + n\left( 1+ {\mathbb E}\left[\frac{p_0}{1-p_0}\right]\right).
$$
We recall that $\gamma= 1+{\mathbb E}\left[\frac{p_0}{1-p_0}\right]$
and we define $(U_n)_n$ such that
$$U_n:=\max\{k\ge 0\, :\, T_k\le n\} .$$
We notice that the sequences of processes
$$ \left(n^{-1}
  \sum_{k=0}^{\lfloor n t\rfloor-1}\left(\xi_k- \frac{p_{S_k}}{1-p_{S_k}} \right),t\geq 0\right)_n
\ \ \mbox{and}\ \
\left(n^{-1} \sum_{k=0}^{\lfloor n t\rfloor-1}\left( \frac{p_{S_k}}{1-p_{S_k}}
       - {\mathbb E}\left[\frac{p_0}{1-p_0}\right] \right),t\geq 0
\right)_n
$$
converge in distribution in $\mathcal{D}([0,+\infty),\RR)$ to 0.
The first convergence follows from Lemma \ref{cvps} where we take $\e_k=1$ for every
$k\in\Z$. The second convergence is a consequence of
\cite[Thm 1.1]{KestenSpitzer} since $n^{\delta}/n\to 0$ as $n\to+\infty$.
Hence $\left(n^{-1}T_{\lfloor n t \rfloor},t\geq 0 \right)_n$ converges in distribution
to $(\gamma t)_t$,
%and $\left(n^{-1}U_{\lfloor n t \rfloor},t\geq 0\right)_n$ converges almost surely to $(1/\gamma$.
%   and to the strong law of large number
%for random walks in random sceneries \cite{},
We conclude that $\left(\left(n^{-1} U_{\floor{nt}}\right)_{t\geq 0}\right)_n$
converges in distribution (in ${\mathcal D}([0;+\infty),{\mathbb R})$)
to $(t/\gamma)_t$.
%(as we did for $\left(\left(n^{-3/4} E_{\floor{nt}}   \right)_{t\ge 0}\right)_n$
%in the beginning of the proof of Proposition \ref{sec:TLF}).
Therefore, according to Proposition \ref{sec:TLF} and to \cite[Lem p. 151, Thm 3.9]{Bil},
the sequence of processes
$$\left(\left(n^{-\delta}X_{U_{\floor {nt}}},n^{-1/2}S_{U_{\floor {nt}}}\right)_{t\ge 0}\right)_n$$
converges in distribution (in ${\mathcal D}([0;+\infty),{\mathbb
  R}^2)$) to
$(\sigma\Delta_{\frac t\gamma},B_{\frac t\gamma})_{t\ge 0}$. This means that
$$\left(\left(n^{-\delta}M^{(1)}_{T_{U_{\floor {nt}}}},n^{-1/2}
   M^{(2)}_{T_{U_{\floor {nt}}}} \right)_{t\ge 0}\right)_n$$
converges in distribution (in ${\mathcal D}([0;+\infty),{\mathbb
  R}^2)$) to
$(\sigma\Delta_{\frac t\gamma},B_{\frac t\gamma})_{t\ge 0}$.

Moreover, we have $B_{\frac t\gamma}=\gamma^{-1/2}B'_t$ and $Z_{\frac x{\sqrt{\gamma}}}
  = \gamma^{-1/(2\beta)}Z'_x$, where $(B'_t)_{t\geq 0}$ is a standard Brownian motion, and
$(Z'_x)_{x\in\mathbb R}$ has the same distribution as $(Z_x)_{x\in\mathbb R}$ and
is independent of $(B'_t)_{t\geq 0}$.
Furthermore we have
$$L_{\frac t\gamma}(x)=\gamma^{-1/2}L'_t(\gamma^{1/2}x),\qquad t\geq 0,\ x\in\RR,$$
where $(L'_t)_{t\geq 0}$ is the local time of $(B'_t)_t$ and so
$$\Delta_{\frac t\gamma}=\gamma^{-\frac 12}\int_{\mathbb R}
      L'_t(\gamma^{\frac 12}x)\, dZ_x=\gamma^{-\delta}\int_{\mathbb R}
      L'_t(y)\, dZ'_y.$$
Hence
$(\sigma\Delta_{\frac t\gamma},B_{\frac t\gamma})_{t\ge 0}$
has the same distribution as $(\sigma\gamma^{-\delta}
    \Delta_t,\gamma^{-\frac 12}B_t)_{t\ge 0}$.

Now we observe that we have
$$ M^{(2)}_{\floor{nt}}=
   M^{(2)}_{T_{U_{\floor {nt}}}}\ \ \mbox{and}\ \
    \left| M^{(1)}_{\floor{nt}}-
   M^{(1)}_{T_{U_{\floor {nt}}}}\right|\le \xi_{U_{\floor{nt}}} $$
and that for every $\theta>0$ and $T>0$,
\begin{eqnarray}
{\mathbb P}\left(\sup_{t\in[0;T]}n^{-\delta}\xi_{U_{\floor{nt}}}
       \ge \theta\right)&\le&
   \sum_{k=0}^{nT}{\mathbb P}\left( \xi_{k}\ge \theta n^{\delta}\right)\nonumber\\
   &\le& \sum_{k=0}^{nT}\frac{{\mathbb E}[(\xi_0)^{\beta-\eta}]}{(\theta n^\delta)^{\beta-\eta}}= o(1)
\label{eqEntre2sauts}
\end{eqnarray}
for $\eta>0$ small enough, since $\delta\beta>1$
and since \footnote{This comes from the H\"older inequality since
${\mathbb E}\left[(\xi_0)^2|p_0\right]\le \frac 2{(1-p_0)^2}$.}
%To prove this, we can notice, for example, that conditionally on $p_0$,
%$\xi_0$ is stochastically dominated by an exponential variable $V$ of parameter $-\ln p_0$
%since
%$\lfloor V\rfloor=_{law}\xi_0$, and that $\EE(V^{\beta-\eta}\vert p_0)=\Gamma(\beta-\eta+1)
%/(-\ln p_0)^{\beta-\eta}$.
(if $\eta<\beta-1$)
${\mathbb E}[(\xi_0)^{\beta-\eta}\vert p_0]\leq C \frac{1}{(1-p_0)^{\beta-\eta}}$ a.s.
and $\EE\left[\frac{1}{(1-p_0)^{\beta-\eta}}\right]<\infty$.
This completes the proof of Theorem \ref{sec:TL}.
\end{proof}

%%%%%%%%%%%%%%%%%%%%%%%%%%%%%%%%%%%%%%%%%%%%%%%%%%%%%%%%%%%%%%%%%%%%%%%%%%%%%%%%%%%%%%%%%%%%%%%%%%%%%%%

%%%%%%%%%%%%%%%%%%%%%%%%%%%%%%%%%%%%%%%%%%%%%%%%%%%%%%%%%%%%%%%%%%%%%%%%%%%%%%%%%%%%%%%%%%%%%%%%%%%%%%%

\section{Proof of Lemma \ref{LemmaCVdansDenDimension2}}\label{SectionPreuvDuLemmeCV}
The proof is very similar to those in \cite{KestenSpitzer} and \cite{GLN2}, with some adaptations.
%; we give it for the sake of completeness.

We define $\tilde D_n:=\sum_{y\in\mathbb Z}\tilde\e_y N_{n}(y), \ n\in\NN$.

\begin{lem}\label{LemmaCVFiniDimensionnelle}
If $\beta\in(1,2]$,
the finite dimensional distributions of $(\tilde D_{\lfloor n t\rfloor}/n^{\delta},
S_{\lfloor n t \rfloor}/\sqrt{n})_{t\geq 0}$ converge to those of $(\tilde\Delta_t, \tilde B_t)_{t\geq 0}$.
\end{lem}

Before proving Lemma \ref{LemmaCVFiniDimensionnelle}, we first introduce some preliminary results.

We observe that
$n^{-1/\beta}\sum_{y=1}^{n}\tilde\e_y$  converges in distribution
to a stable random variable of parameter $\beta$,
with characteristic function $\widetilde{\zeta}_\beta(\theta):=\exp(-A_0|\theta|^{\beta})$
(for some $A_0>0$).
We can now compute the characteristic function of the finite dimensional distributions of
$(\tilde\Delta_t, \tilde B_t)_{t\geq 0}$.

\begin{lem} \label{LemmaCharacteriticFunctionFiniteDimensional}
Let $k\in\NN^*$, $(t_1, t_2,\dots, t_k)\in\RR_+^k$ and $(\theta_1^{(i)}, \theta_2^{(i)},\dots, \theta_k^{(i)})_{i=1,2}\in\RR^{2k}$. We have,
\begin{eqnarray}
\EE\left[\exp\left(i\sum_{j=1}^k \big(\theta_j^{(1)}\tilde\Delta_{t_j}+\theta_j^{(2)}\tilde B_{t_j}\big)
\right)\right]&=&
\EE\left[
\exp\left(-A_0\int_{-\infty}^{+\infty}\left|\sum_{j=1}^k \theta_j^{(1)}\tilde L_{t_j}(x)
\right|^\beta d x\right)
\exp\left(i\sum_{j=1}^k \theta_j^{(2)}\tilde B_{t_j}\right)
\right].\nonumber\\
&&\label{eqCharacteriticFunctionFiniteDimensional}
\end{eqnarray}
\end{lem}

\begin{proof} We condition by $\tilde B$ and we proceed as in
\cite[Lem 5]{KestenSpitzer}. We get
\begin{equation*}
\EE\left[\left.\exp\left(i\sum_{j=1}^k \theta_j^{(1)}\tilde\Delta_{t_j}
\right)\right|\tilde B\right]=
\exp\left(-A_0\int_{-\infty}^{+\infty}\left|\sum_{j=1}^k \theta_j^{(1)}\tilde
    L_{t_j}(x)\right|^\beta d x\right),
\end{equation*}
which gives the result.
\end{proof}

For fixed
$k\in\NN^*$ and $(t_1, t_2,\dots, t_k)\in\RR_+^k$, we define for every $(\theta_1, \theta_2,\dots, \theta_k)\in(\RR^{2})^k$,
\begin{equation*}
\psi_n(\theta_1,\theta_2,\dots,\theta_k)
:=    \EE\left[\exp\left(-A_0\sum_{y\in\Z}\left|\sum_{j=1}^k\theta_j^{(1)}N_{\lfloor n t_j\rfloor}(y)n^{-\delta}\right|^\beta\right)
    \exp\left(i\sum_{j=1}^k \theta_j^{(2)}\frac{S_{\lfloor n t_j\rfloor}}{\sqrt{n}}\right)
    \right]
\end{equation*}
and
\begin{eqnarray*}
\phi_n(\theta_1,\theta_2,\dots,\theta_k)
&:=& \EE\left[
    \exp\left(i\sum_{j=1}^k \big(\theta_j^{(1)}n^{-\delta}\tilde D_{\lfloor n t_j\rfloor}
    +\theta_j^{(2)}\frac{S_{\lfloor nt_j\rfloor}}{\sqrt{n}}\big)
    \right)\right]\\
& = &
    \EE\left[\prod_{y\in\Z}\tilde{\lambda}\left(\sum_{j=1}^k\theta_j^{(1)}
   N_{\lfloor n t_j\rfloor}(y)n^{-\delta}\right)
    \exp\left(i\sum_{j=1}^k \theta_j^{(2)}\frac{S_{\lfloor n t_j\rfloor}}{\sqrt{n}}\right)
    \right]
\end{eqnarray*}
where $\tilde{\lambda}(\theta):=\EE[\exp(i\theta \tilde\e_0)]$ for every $\theta\in\RR$
and $\theta_j=(\theta_j^{(1)},\theta_j^{(2)})$ for every $j\in\{1,\dots,n\}$.

\begin{lem}
For every $k\in\NN^*$, $(t_1, t_2,\dots, t_k)\in\RR_+^k$ and  $(\theta_1, \theta_2,\dots, \theta_k)\in(\RR^{2})^k$,
\begin{equation*}\label{LienPsiPhi}
\lim_{n\to+\infty}|\psi_n(\theta_1,\theta_2,\dots,\theta_k)
-\phi_n(\theta_1,\theta_2,\dots,\theta_k)
|=0.
\end{equation*}
\end{lem}

\begin{proof}
As in \cite[p. 7]{KestenSpitzer}, we have $1-\tilde{\lambda}(\theta)\sim_{\theta\to 0} A_0
|\theta|^\beta$ since the distribution of $\tilde\e_0$ belongs to the normal domain of
attraction of
the stable distribution with characteristic function $\widetilde{\zeta}_\beta$. The remainder
of the proof is the same as in \cite[Lem 5]{GLN2} with $\delta$ instead of $3/4$ and $\beta$
instead of $2$, since $\PP(n^{-\delta} \sup_{y\in\Z}|\sum_{j=1}^k\theta_j^{(1)}
N_{\lfloor n t_j\rfloor}(y)|>\e)\to_{n\to+\infty}0$ for $\e>0$ by \cite[Lem 4]{KestenSpitzer}
and since we have $\EE(\sum_{y\in\Z}|  n^{-\delta}\sum_{j=1}^k \theta_j^{(1)} N_{\lfloor n t
 _j\rfloor}(y) |^\beta)\leq C<\infty$ by \cite[Lem 3.3]{DGP}.
\end{proof}

We now prove

\begin{lem}\label{LemmeCVjointe}
For every $k\in\NN^*$, $(t_1, t_2,\dots, t_k)\in\RR_+^k$ and $(\theta_1, \theta_2,\dots, \theta_k)\in(\RR^{2})^k$,
\begin{equation*}
\left(
n^{-\delta \beta}\sum_{y\in\Z}\left|\sum_{j=1}^k \theta_j^{(1)}N_{\lfloor n t_j\rfloor}(y)\right|^\beta,
\sum_{j=1}^k\theta_j^{(2)} S_{\lfloor n t_j\rfloor}/\sqrt{n}
\right)_n
\end{equation*}
converges in distribution as $n\to+\infty$ to
\begin{equation*}
\left(\int_{-\infty}^{+\infty}\left|\sum_{j=1}^k \theta_j^{(1)}\tilde L_{t_j}(x)\right|^\beta d x ,
\sum_{j=1}^k \theta_j^{(2)} \tilde B_{t_j} \right).
\end{equation*}

\end{lem}

\begin{proof}
The proof is very similar to the one of \cite[Lem 6]{GLN2}, and to the proof of \cite[Lem 6]{KestenSpitzer} which deals with the first coordinate.
Throughout the proof, $C$ denotes a positive constant, which can vary from line to line, and can
depend on $(\theta_j^{(i)}, i=1,2;j=1,\dots,k)$.
%If $y\in\Z$, $\ell\in\Z$ and $\ell< t<\ell+1$, we define $N_t(y):=(\ell+1-t)N_k(y)+(t-\ell)N_{k+1}(y)$. Moreover,
For $n\in\NN$ and real numbers $a<b$ and $t>0$, we introduce the notation
\begin{equation*}
T_{t}^n(a, b):=\int_0^{\lfloor n t\rfloor/n} 1_{\{a\leq S_{\lfloor ns\rfloor}/\sqrt{n}< b\}} d s,
\end{equation*}
which is the occupation time of $[a,b)$ by $S_{\lfloor n .\rfloor}/\sqrt{n}$ up to time $\lfloor n t\rfloor/n$.
We consider $\tau>0$ and two real numbers $\mu_1$ and $\mu_2$. We define for $M>0$, $\ell\in\ZZ$ and $n\in\NN$,
\begin{eqnarray*}
U(\tau, M, n) & := &
    \mu_1 n^{-\delta\beta}\sum_{\tiny \begin{array}{c}y<-M\tau\sqrt{n}\\ \text{or } y\geq M\tau\sqrt{n}\end{array}}
    \left|\sum_{j=1}^k \theta_j^{(1)}N_{\lfloor n t_j\rfloor} (y)\right|^\beta,\\
T(\ell,n) & := &
    \sum_{j=1}^k \theta_j^{(1)}T_{t_j}^n(\ell\tau, (\ell+1)\tau)
    =\frac{1}{n}\sum_{j=1}^k \theta_j^{(1)}\sum_{\ell\tau\sqrt{n}\leq y < (\ell+1)\tau\sqrt{n}}
    N_{\lfloor  n t_j\rfloor }(y),\\
V(\tau, M,n) & := &
    \mu_1\tau^{1-\beta}\sum_{-M\leq \ell<M}|T(\ell, n)|^\beta
    +\mu_2 n^{-1/2}\sum_{j=1}^k \theta_j^{(2)}S_{\lfloor n t_j\rfloor}.
\end{eqnarray*}
We are interested in
\begin{eqnarray*}
A(\tau, M,n) & := &
    \frac{\mu_1}{n^{\delta \beta}}\sum_{y\in\Z}\left|\sum_{j=1}^k \theta_j^{(1)}N_{\lfloor n t_j \rfloor}(y)\right|^\beta
    +\mu_2\sum_{j=1}^k \theta_j^{(2)}\frac{S_{\lfloor n t_j\rfloor}}{\sqrt{n}}
    -U(\tau, M, n)-V(\tau, M,n)\\
& = &
    \frac{\mu_1}{n^{\delta\beta}}\sum_{-M\leq\ell<M}\sum_{\ell\tau\sqrt{n}\leq y<(\ell+1)\tau\sqrt{n}}
    %\left[
    \left|\sum_{j=1}^k\theta_j^{(1)}N_{\lfloor n t_j \rfloor}(y)\right|^\beta
    %-\frac{n^{\delta \beta}}{c(\ell,n)}
    -\mu_1\sum_{-M\leq\ell<M}
    \tau^{1-\beta}|T(\ell,n)|^\beta
    %\right]
    .
\end{eqnarray*}
{\it First step:}
We  define $c(\ell,n):=\#\{y\in\Z,\ \ell\tau\sqrt{n}\leq y<(\ell+1)\tau\sqrt{n}\}$.
As in \cite{KestenSpitzer}, we have for $\mu_1\neq 0$,
\begin{eqnarray}
&&\mu_1^{-1}A(\tau, M,n) \nonumber\\
& = &
    \sum_{-M\leq\ell<M}\sum_{\ell\tau\sqrt{n}\leq y<(\ell+1)\tau\sqrt{n}}
    n^{-\delta\beta}\left[\left|\sum_{j=1}^k\theta_j^{(1)}N_{\lfloor n t_j \rfloor}(y)\right|^\beta
    -n^{\beta}(\tau\sqrt{n})^{-\beta}|T(\ell,n)|^\beta
    \right]
\label{LigneduHaut}
\\
& + &
    \sum_{-M\leq \ell <M}\left[n^{\beta-\delta\beta}(\tau\sqrt{n})^{-\beta}c(\ell,n)-\tau^{1-\beta}\right]|T(\ell,n)|^\beta.
\label{LigneduBas}
\end{eqnarray}
As in \cite[p. 19]{KestenSpitzer}, the right hand side of \eqref{LigneduBas} tends to $0$ in probability as $n\to+\infty$. Then we just have to study \eqref{LigneduHaut}. To this aim, we use the inequality suggested by \cite{KestenSpitzer}, that is
\begin{equation*}
\forall (a,b)\in\RR_+^2,\qquad |a^\beta-b^\beta|\leq \beta|a-b|(a^{\beta-1}+b^{\beta-1})
\leq 2\beta|a-b|(a+b)^{\beta-1}
\end{equation*}
since $\beta>1$. We define $T'(\ell,n)$ by the same formula as $T(\ell,n)$ where we replace each $\theta_j^{(1)}$ by   $|\theta_j^{(1)}|$. We consider, for $\ell\tau\sqrt{n}\leq y<(\ell+1)\tau\sqrt{n}$,
\begin{eqnarray}
&&
\EE\left(\left|
    \left|\sum_{j=1}^k\theta_j^{(1)}N_{\lfloor n t_j \rfloor}(y)\right|^\beta
    -n^{\beta}(\tau\sqrt{n})^{-\beta}|T(\ell,n)|^\beta
    \right|\right)
\label{EqEsperanceduHaut}
\\
& \leq &
    2\beta
    \EE\left(\left|
    \left|\sum_{j=1}^k\theta_j^{(1)}N_{\lfloor n t_j \rfloor}(y)\right|
    -\frac{\sqrt{n}}{\tau}|T(\ell,n)|
    \right|
.
    \left|
    \sum_{j=1}^k|\theta_j^{(1)}|N_{\lfloor n t_j \rfloor}(y)
    +\frac{\sqrt{n}}{\tau}T'(\ell,n)
    \right|^{\beta-1}\right)
\nonumber\\
& \leq &
    2\beta
    \EE\left(
    \left|\sum_{j=1}^k\theta_j^{(1)}N_{\lfloor n t_j \rfloor}(y)
    -\frac{\sqrt{n}}{\tau}T(\ell,n)
    \right|^2\right)^{1/2}
    \label{eqHolderavecP}
    \\
&&
\times
    \EE\left(
    \left|
    \sum_{j=1}^k|\theta_j^{(1)}|N_{\lfloor n t_j \rfloor}(y)
    +\frac{\sqrt{n}}{\tau}T'(\ell,n)
    \right|^{2(\beta-1)}
    \right)^{1/2}
    \label{eqHolderavecQ}
\end{eqnarray}
by the Cauchy-Schwarz inequality
and by the second triangular inequality in \eqref{eqHolderavecP}.
%with $p:=\frac{2}{3-\beta}>1$ and $q:=\frac{2}{\beta-1}>1$ since $\beta\in(1,2)$.
In the following RHS will stand for right hand side.
We have by \cite{KestenSpitzer} equations (3.9) and (2.26),
\begin{equation}\label{InegHolderTerme1}
\text{RHS of } \eqref{eqHolderavecP}\leq (C_1 \tau n)^{1/2},
\end{equation}
where $C_1$ is a constant, which is finite since $gcd\{k\ :\ \PP(S_1=k)>0\}=1$.
Moreover, setting $a(\ell,n):=\tau\ell\sqrt{n}$,
by the H\"{o}lder inequality and \cite[p. 346]{GLN2}, we have
\begin{eqnarray}
[\text{RHS of }\eqref{eqHolderavecQ}]^{\frac{2}{\beta-1}}
& \leq &
    \EE\left(
    \left|
    \sum_{j=1}^k|\theta_j^{(1)}|N_{\lfloor n t_j \rfloor}(y)
    +\frac{\sqrt{n}}{\tau}T'(\ell,n)
    \right|^{2}
    \right)
\\
& \leq &
C\sum_{j=1}^k
\max_{a(\ell,n)\leq x<a(\ell+1,n)}
(
\EE(N_{\lfloor n t_j\rfloor}(x)^3)^{2/3}
+\EE(N_{\lfloor n t_j\rfloor}(y)^3)^{2/3}
)
\nonumber\\
& \leq &
C \EE(N_{\lfloor n \max(t_1,\dots,t_k)\rfloor}(0)^3)^{2/3}
\leq
C n\label{InegHolderTerme2}
\end{eqnarray}
by \cite[Lem 1]{KestenSpitzer}. Combining \eqref{InegHolderTerme1} and \eqref{InegHolderTerme2}, we get
\begin{equation*}
\text{RHS of }\eqref{EqEsperanceduHaut}\leq C \tau^{\frac{1}{2}}n^{\frac{\beta}{2}}.
\end{equation*}
Hence,
\begin{equation*}
\EE(| \text{RHS of }\eqref{LigneduHaut}|)\leq C(2M+1) \tau^{\frac{3}{2}}.
\end{equation*}
As in \cite[p. 20]{KestenSpitzer}, for each $\eta>0$ we can take $M\tau$ so large that
\begin{equation}\label{ProbadeU}
\PP(U(M,n,\tau)\neq 0)\leq \eta
\end{equation}
and then $\tau$ so small that
\begin{equation}\label{InegEtaCarre}
\EE(| \text{RHS of }\eqref{LigneduHaut}|)\leq \eta^2/|\mu_1|.
\end{equation}
Hence, by \eqref{ProbadeU} and \eqref{InegEtaCarre}, and since the right hand side of \eqref{LigneduBas} tends to $0$ in probability as $n\to+\infty$, we get for $n$ large enough (even when $\mu_1=0$),
\begin{equation*}
\PP(|A(\tau, M,n)+U(\tau, M,n)|>3\eta)
\leq \PP(|A(\tau, M,n)|>3\eta)+\PP(U(M,n,\tau)\neq 0)
\leq 3\eta.
\end{equation*}

\noindent{\it Second step:} As in \cite[Lem 2]{GLN2}, we have
\begin{equation*}
(T_{t_j}^{(n)}(\ell\tau,(\ell+1)\tau), S_{\lfloor n t_j \rfloor}/\sqrt{n})_{j=1,\dots,k,\ell=-M,\dots,M}
\rightarrow
(\Lambda_{t_j}(\ell\tau,(\ell+1)\tau),\tilde B_{t_j})_{j=1,\dots,k,\ell=-M,\dots,M}
\end{equation*}
in distribution,
as $n\to+\infty$, where $\Lambda_t(a,b):=\int_a^b \tilde L_t(x) d x$ for $t>0$ and
$a<b$. Consequently, $(V(\tau,M,n))_n$ converges in distribution as $n\to+\infty$ to
\begin{equation*}
\overline{V}(\tau, M) :=
    \mu_1\tau^{1-\beta}\sum_{-M\leq \ell<M}\left|\sum_{j=1}^k \theta_j^{(1)}\Lambda_{t_j}(\ell\tau,(\ell+1)\tau)\right|^\beta
    +\mu_2\sum_{j=1}^k \theta_j^{(2)}\tilde B_{t_j}.
\end{equation*}
Since $\tilde L_t(.)$ is continuous with a compact support, %for every $A>0$
we get almost surely
\begin{equation*}
\overline{V}(\tau, M)\to_{M\tau\to +\infty,\tau\to 0}\mu_1\int_{-\infty}^{+\infty}
  \left|\sum_{j=1}^k \theta_j^{(1)}\tilde L_{t_j}(x)\right|^\beta d x
+\mu_2\sum_{j=1}^k \theta_j^{(2)}\tilde B_{t_j}=: \widehat{V}.
\end{equation*}
Hence by choosing adequate $M$ and $\tau$ we get for $n$ large enough
\begin{equation*}
\left|\EE\left[\exp\left(
    i\frac{\mu_1}{n^{\delta \beta}}\sum_{y\in\Z}\left|\sum_{j=1}^k \theta_j^{(1)}N_{\lfloor n t_j \rfloor}(y)\right|^\beta
    +i\mu_2\sum_{j=1}^k \theta_j^{(2)}\frac{S_{\lfloor n t_j\rfloor}}{\sqrt{n}}\right)
    \right]-\EE\exp(i\widehat{V})\right|\leq 11
\eta.
\end{equation*}
Since this is true for every $\mu_1\in\RR$, $\mu_2\in\RR$ and $\eta>0$, this proves Lemma \ref{LemmeCVjointe}.
\end{proof}

\noindent {\it Proof of Lemma \ref{LemmaCVFiniDimensionnelle}}.
Applying Lemma \ref{LemmeCVjointe}, we get the convergence of
$\psi_n(\theta_1,\dots,\theta_k)$ to the right hand side of
\eqref{eqCharacteriticFunctionFiniteDimensional} as $n\to+\infty$. This combined with
Lemma \ref{LemmaCharacteriticFunctionFiniteDimensional} and Lemma \ref{LienPsiPhi}
proves Lemma \ref{LemmaCVFiniDimensionnelle}.
\hfill$\Box$

\noindent {\it Proof of Lemma \ref{LemmaCVdansDenDimension2}}.
We now turn to the tightness. We know that $(\tilde D_{\lfloor n t\rfloor}/n^{\delta},\ t\geq 0)_n$ and
$(S_{\lfloor n t \rfloor}/\sqrt{n},\ t\geq 0)_n$ both converge in distribution in
$\mathcal{D}([0,+\infty),\RR)$ to continuous processes (respectively by
\cite{KestenSpitzer} and by the theorem of Donsker), and
the finite dimensional distributions of $(\tilde D_{\lfloor n t\rfloor}/n^{\delta}, S_{\lfloor n t \rfloor}/\sqrt{n})_{t\geq 0}$ converge to those of $(\tilde\Delta_t, \tilde B_t)_{t\geq 0}$ by Lemma
\ref{LemmaCVFiniDimensionnelle}, hence the distributions of $(\tilde D_{\lfloor n t\rfloor}/n^{\delta}, S_{\lfloor n t \rfloor}/\sqrt{n})_{t\geq 0}$ are tight in $\mathcal{D}([0,+\infty),\RR^2)$ (this is a consequence of \cite{Bil} Theorems 13.2 and 13.4, Corollary p.142 and inequalities (12.7) and (12.9)). This proves Lemma \ref{LemmaCVdansDenDimension2}.

{\bf Acknowledgements~:\/}

The authors thank Yves Derriennic for interesting discussions and
references.

%%%%%%%%%%%%%%%%%%%%%%%%%%%%%%%%%%%%%%

%%               BIBLIOGRAPHIE

%%%%%%%%%%%%%%%%%%%%%%%%%%%%%%%%%%%%%%

\end{document}